\documentclass[11pt]{article}
\usepackage{makeidx}
\usepackage{amssymb}
\usepackage{amsfonts}
\usepackage{amsmath}
\usepackage{amsthm}
\usepackage[utf8]{inputenc}
\usepackage{geometry}
\usepackage{lipsum}
\usepackage{nicefrac}
\usepackage{graphicx}
\usepackage{epstopdf}
\usepackage[dvipsnames]{color}
\usepackage{url}
\usepackage{algpseudocode}
\usepackage{algorithm}
\usepackage{cleveref}

\newtheorem{theorem}{Theorem}
\newtheorem{proposition}[theorem]{Proposition}

\theoremstyle{definition}

\newcommand{\A}{\mathcal{A}}

\newcommand{\balpha}{\bar{\alpha}}

\newcommand{\bbeta}{b_{\beta}}

\newcommand{\bones}{\mathbf{1}}
\newcommand{\bzeros}{\mathbf{0}}

\newcommand{\br}{\bar{r}}

\newcommand{\bz}{\bar{z}}

\newcommand{\covM}{C_{\eta}}
\newcommand{\dg}{\mbox{diag}}

\newcommand{\eps}{\varepsilon}

\newcommand{\K}{\mathcal{K}}
\newcommand{\Krym}{\K_m(\widetilde{A},\widetilde{r}_0)}

\newcommand{\mmax}{m_{\max}}
\newcommand{\mmaxin}{\mmax^{\text{in}}}
\newcommand{\mmaxout}{k_{\max}^{\text{out}}}

\newcommand{\Norm}{\mbox{Normal}}

\newcommand{\Poiss}{\mbox{Poisson}}
\newcommand{\R}{\mathbb{R}}

\newcommand{\spn}{\mbox{span}}
\newcommand{\tql}{\textquotedblleft}
\newcommand{\tqr}{\textquotedblright}

\newcommand{\weps}{\widetilde{\varepsilon}}

\newcommand{\xex}{x^{ex}}

\title{Fast nonnegative least squares\\through flexible Krylov subspaces\thanks{This work was supported by the UK Engineering and Physical Sciences Research Council (EPSRC, grants EP/M019306/1 and EP/M008843/1)}}

\author{Silvia Gazzola
\thanks{Institute of Sensors, Signals and Systems. School of Engineering and Physical Sciences. Heriot-Watt University, Edinburgh, UK. (\texttt{S.Gazzola, Y.Wiaux@hw.ac.uk}).}
\and Yves Wiaux\footnotemark[2]}

\begin{document}
\maketitle

\begin{abstract}
Constrained least squares problems arise in a variety of applications, and many iterative methods are already available to compute their solutions. This paper proposes a new efficient approach to solve nonnegative linear least squares problems. The associated KKT conditions are leveraged to form an adaptively preconditioned linear system, which is then solved by a flexible Krylov subspace method. The new method can be easily applied to image reconstruction problems affected by both Gaussian and Poisson noise, where the components of the solution represent nonnegative intensities. {Theoretical insight is given, and} numerical experiments and comparisons are displayed in order to validate the new method, which delivers results of equal or better quality than many state-of-the-art methods for nonnegative least squares solvers, with a significant speedup.
\end{abstract}

\section{Introduction}

Let us consider the constrained linear least squares problem
\begin{equation}\label{NNLS}
\min_{x\geq0}\Phi(x)\,,\quad\Phi(x):=\left\Vert b-Ax\right\Vert_2^2\,,
\end{equation}
where $A\in\R^{M\times N}$, and the constraint $x\geq0$ on $x\in\R^N$ is intended component-wise. This problem typically arises in imaging applications, where the entries (pixels) of $x$ represent light intensities, which are nonnegative. Common examples are 2D image deblurring and reconstruction problems. When dealing with the former, one seeks to restore a blurred image. The 
blur is assumed to be known, and it depends on the application; for instance, when dealing with astronomical images acquired by 
ground-based telescopes, the blur is caused by atmospheric turbulence. When dealing with image reconstruction problems, the goal is to compute the image of an object given a set of projections {thereof}. Since the vector $b\in\R^M$ in (\ref{NNLS}) represents collected measurements, it is usually affected by (unknown) random noise $\eta$, i.e., $b=b^{ex}+\eta$. Image {deblurring and} reconstruction problems are also ill-posed, and some regularization should be considered in order to compute a meaningful approximation of the solution $x^{ex}$ of the consistent exact system
\begin{equation}\label{exact}
Ax^{ex}=b^{ex},
\end{equation} 
cf. \cite{PCH10}. The nonnegativity constraints in (\ref{NNLS}) can be considered as a form of regularization, since information about the problem is incorporated into the model. Nonnegatively constrained least squares also arise naturally in the class of Fourier imaging applications, where the measurements are acquired in the Fourier domain. Key examples are magnetic resonance imaging in medicine \cite{cMRI}, and radio-interferometric imaging in astronomy \cite{SARA}, which {typically} require the solution of under-determined linear systems. Another strong information to enforce into imaging applications is sparsity, i.e., often only a small number of pixels (or, more in general, expansion coefficients with respect to a sparsity basis) are nonzero. The authors of \cite{Elad} prove that, if the measurement matrix $A$ in (\ref{exact})  satisfies some assumptions and if $\xex$ is sufficiently sparse, then requiring nonnegativity is equivalent to requiring sparsity (e.g., minimizing the $\ell_1$ norm of $x$). 
Because of the multi-dimensional nature of imaging problems, the quantities in (\ref{NNLS}) are typically large-scale. In particular, with the advent of next-generation radio telescopes such as the Square Kilometre Array, solvers for (\ref{NNLS}) must be able to handle data sizes of the order of Terabytes or more{: this implies that the employed algorithms should be scalable and extremely efficient, with an overall low computational cost, usually measured as total number of matrix-vector products.} 

In developing our theory, and in our experiments, 
{problems affected by Gaussian white noise are mainly taken into account. Some possible extensions to include problems affected by both Gaussian and Poisson noise are outlined}. The latter is a realistic model when dealing with astronomical images acquired by charge-coupled-devices (cf. \cite{BN06}, and the references therein). Following the derivations in \cite{BN06}, the generic noise model is 
given by
\begin{equation}\label{NoiseModel}
b = \Poiss(A\xex)+\Poiss(\beta\bones)+\Norm(\bzeros,\sigma^2I)\,,
\end{equation}
where $\bones=[1,\dots,1]^T\in\R^{M}$, {$\bzeros=[0,\dots,0]^T\in\R^{M}$}, and $I$ is the identity matrix of order $M$. After some statistical approximations, the noisy problem associated with (\ref{exact}) can be expressed as
\begin{equation}\label{GaussPoissPb}
\underbrace{b-\beta\bones}_{=:\,\bbeta}=A\xex+\Norm(\bzeros,\underbrace{\dg(A\xex+\beta\bones+\sigma^2\bones)}
_{=:\,\covM})\,,
\end{equation}
where the terms $A\xex$ and $\beta\bones$ in the above random variables originate from the Poisson terms on the right-hand side of (\ref{NoiseModel}).
The entries of $A$ are assumed to be nonnegative (this is typically true for imaging problems), so that the diagonal elements of $\covM$ in (\ref{GaussPoissPb}) are positive.
In particular, when only Gaussian noise is involved, equation (\ref{GaussPoissPb}) reduces to
\begin{equation}\label{GaussPb}
b=A\xex+\Norm(\bzeros,\sigma^2I)\,,
\end{equation}
and the associated nonnegative problem is (\ref{NNLS}). In general, the nonnegative least squares problem associated with the formulation (\ref{GaussPoissPb}) reads as
\begin{equation}\label{covNNLS}
\min_{x\geq 0}\Phi_{C}(x),\quad\Phi_{C}(x):=\|\covM^{-1/2}(\bbeta-Ax)\|_2^2\,,
\end{equation}
where 
a power of the covariance matrix is basically incorporated as left preconditioner for (\ref{NNLS}). Note that, since $\covM$ is a function of $A\xex$, one should consider an approximation (cf. again \cite{BN06}).

\subsection{{Related works}} 
\label{sect:survey}
Over the last decades, many methods have been derived to iteratively solve problem (\ref{NNLS}). 
The most basic ones are the so-called {gradient} projection methods that, at each iteration, combine a descent step in the direction of the negative gradient of $\Phi(x)$ 
with a projection onto the nonnegative orthant. Some examples include the projected Landweber (or Richardson) method, the projected Cimmino method, and the projected Steepest Descent method {(NNSD)}: these methods differ in the way the step-size is set, and in the possible use of fixed preconditioners (cf. \cite{BerNagy13} and the references therein). In the optimization literature, these methods {can be incorporated into the framework of the so-called} \tql iterative shrinkage thresholding algorithms\tqr\ (ISTA), {and a well-known accelerated version of them, dubbed \tql FISTA\tqr\ \cite{FISTA}, will be considered in the following sections.} 

Another class of methods for the solution of nonnegative least squares stems from the enforcement of the Karush-Kuhn-Tucker (KKT) conditions.
For problem (\ref{NNLS}), the KKT conditions are necessary and sufficient for optimality, and they can be compactly expressed as
\begin{equation}\label{KKT}
X A^T(A x-b)=0\,,\quad\mbox{where}\quad X=\dg(x)\,,\quad x\geq 0\,,\quad{A^T(A x-b)\geq0\,,}
\end{equation}
cf. \cite[\S 6.8]{LS} and \cite{NS00}. 
The authors of \cite{HNV00, NS00} remark that 
{the conditions in} (\ref{KKT}) can be also obtained by reformulating (\ref{NNLS}) as a convex unconstrained least squares problem with the component-wise re-parametrization $x=e^z$, and by imposing the stationarity condition on the gradient of the objective function computed by the chain rule: this is immediate from the fact that 
\begin{equation}\label{ChainRule}
\nabla_z\Phi(x)=\dg(x)\nabla_x\Phi(x)=XA^T(A x-b)\,.
\end{equation}
However, this is not the point of view adopted in this paper. 
It should be underlined that {the first condition in} (\ref{KKT}) is a nonlinear system of equations with respect to $x$, and different approaches for its solution are already available in the literature, cf. \cite{BN06, Kaufman, NS00}. All of them can be expressed as fixed-point iterations of the form 
\begin{equation}\label{MRNSD}
x_{m}=x_{m-1}+\alpha_{m-1} \underbrace{X^{(m)}A^T(b-A x_{m-1})}_{=:\,d_{m-1}}\,,\quad\mbox{where}\quad 
X^{(m)}=\dg(x_{m-1})\,.
\end{equation}
Note that the vector $d_{m-1}$ is the negative gradient $-\nabla_z\Phi(x)$ computed in $x_{m-1}$, so that these methods still descend along the direction of the gradient. However, with respect to the usual gradient descent methods applied to solve unconstrained least squares problems, the step length $\alpha_{m-1}$ in the above equation should be somewhat bounded in order to impose 
nonnegativity of the approximate solution at each iteration. For this reason, in \cite{BN06, NS00} these methods are named Modified Residual Norm Steepest Descent (MRNSD) algorithms. 

When considering problems affected by both Gaussian and Poisson noise, a scheme similar to (\ref{MRNSD}) can be applied to solve problem (\ref{covNNLS}). Namely, 
the iterates are updated as
\begin{equation}\label{covMRNSD}
x_{m}=x_{m-1}+\alpha_{m-1} X^{(m)}A^T\covM^{-1}(b_{\beta}-A x_{m-1})\,,\quad\mbox{with}\quad 
X^{(m)}=\dg(x_{m-1})\,,
\end{equation}
and different methods originate from different approximations of the diagonal covariance matrix $\covM$. The authors of \cite{BN06} outline two strategies: the first one consists {of} taking 
\begin{equation}\label{covMfix}
\covM = \dg(b+\sigma^2\bones)\,,
\end{equation}
and the corresponding method (\ref{covMRNSD}) is called weighted MRNSD (WMRNSD) algorithm; the second one considers a step-dependent $\covM$, defined as
\begin{equation}\label{covMvar}
\covM^{(m)}=\dg(Ax_{m-1}+\beta\bones+\sigma^2\bones)\,,
\end{equation}
and the corresponding method (\ref{covMRNSD}) is called $k$-weighted MRNSD (KWMRNSD). 
Many numerical experiments available in the literature show that the class of the MRNSD methods is efficient and reliable. 
However, as usual when considering gradient descent methods, the rate of convergence is 
quite slow. The authors of \cite{BN06, NS00} use some (additional) specific left preconditioners $L$ to accelerate the convergence of the MRNSD methods {(PMRNSD)}, so that $A$ in (\ref{MRNSD}) and (\ref{covMRNSD}) is replaced by the matrix $L^{-1}A$.

Krylov subspace methods are well-known iterative solvers that are commonly employed to regularize unconstrained least squares problems of the form 
\begin{equation}\label{LS}
\min_{x\in\R^N}\Phi(x)\,,
\end{equation} 
where $\Phi(x)$ is defined as in (\ref{NNLS}), or the normal equations
\begin{equation}\label{NE}
A^TA x = A^Tb
\end{equation}
associated with it, cf. \cite{Survey15, PCH10}. To keep the derivations simpler, only problems affected by Gaussian noise are considered at this stage.
At the $m$th iteration of a projection method, an approximation $x_m$ of a solution of (\ref{LS}) is computed by imposing 
$x_m$ to belong to an $m$-dimensional approximation subspace $\mathcal{A}_m$, and additional constraints on the corresponding residual $r_m:=b-A x_m$ {or $A^Tr_m$}.
{Given an initial guess $x_0$} for the solution of (\ref{LS}), a Krylov method is a projection method whose approximation subspace $\mathcal{A}_m$ is of the form $x_0+\K_m(\widetilde{A},\widetilde{r}_0)$, where $\K_m(\widetilde{A},\widetilde{r}_0)$ is a Krylov subspace defined as 
\begin{equation}\label{KrylovSp}
\K_m(\widetilde{A},\widetilde{r}_0)=\spn\{\widetilde{r}_0,\widetilde{A}\widetilde{r}_0,\dots,(\widetilde{A})^{m-1}\widetilde{r}_0\}\,,
\end{equation}
and is assumed to be of dimension $m$. Different Krylov methods are obtained by varying the conditions on $r_m$ {or $A^Tr_m$}, the matrix $\widetilde{A}$, and the vector $\widetilde{r}_0$ in $\Krym$: every square matrix linked to $A$, and any vector linked to $r_0$ can be potentially used (cf. \cite[Chaper 6]{Saad}). CG {(Conjugate Gradient)}, CGLS {(CG for Least Squares problems)}, 
GMRES {(Generalized Minimal Residual)}, and RR-GMRES {(Range Restricted GMRES)} are among the most popular Krylov subspace methods. 
Various theoretical considerations and many numerical experiments available in the literature show that Krylov methods are much more efficient than the gradient descent methods for the solution of (\ref{LS}). Indeed, when considering image deblurring problems, some Krylov methods can compute a good regularized solution in only a few iterations, i.e., requiring only a few matrix-vector products with $A$ and/or $A^T$. Unfortunately, Krylov subspace methods cannot be straightforwardly adopted to handle constrained problems in general, and problem (\ref{NNLS}) in particular.

\begin{figure}[htbp]
\centering
\begin{tabular}{cc}
\hspace{-0.6cm}{\small \textbf{(a)}} & \hspace{-0.6cm}{\small \textbf{(b)}}\\ 
\hspace{-0.6cm}\includegraphics[width=5cm]{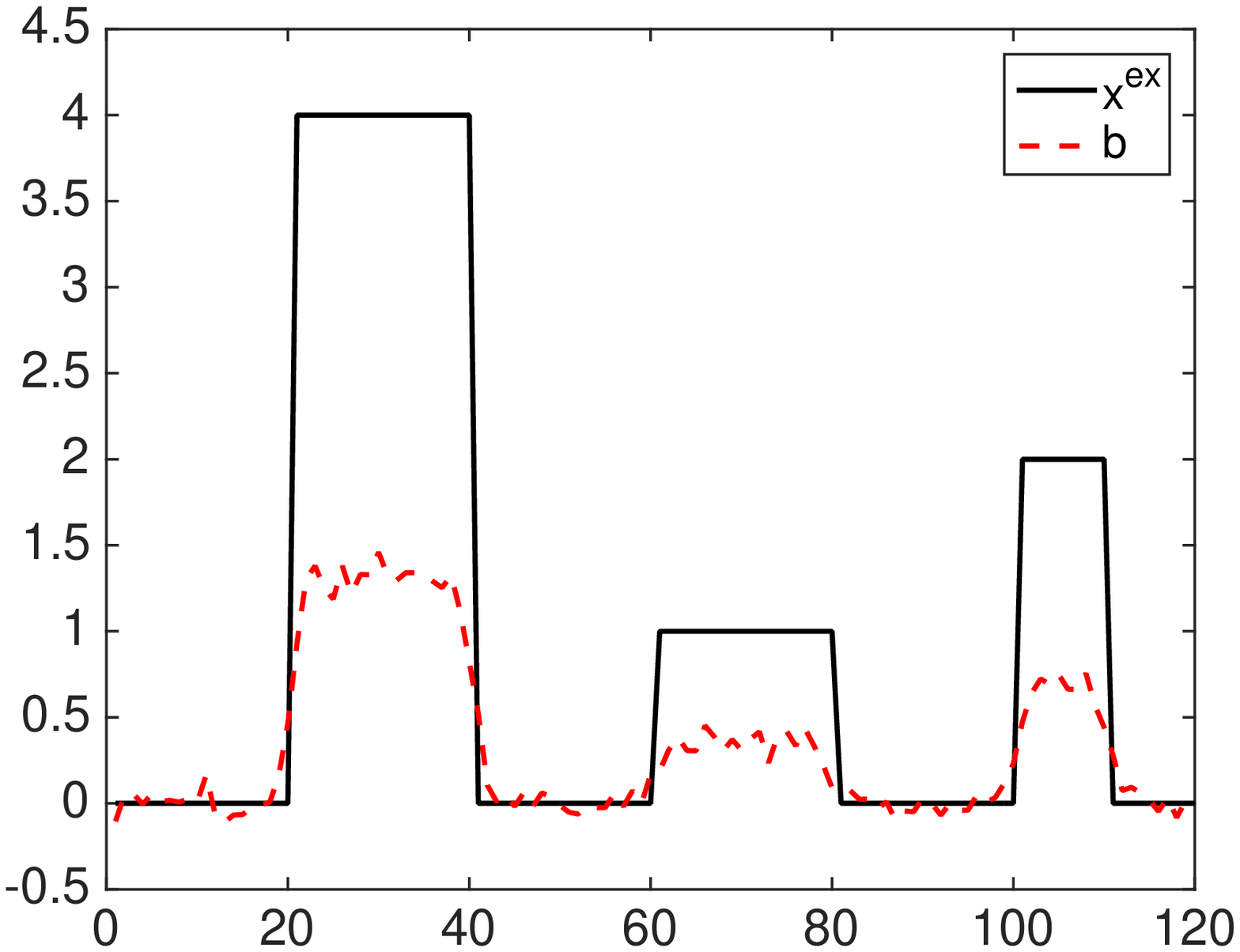} & 
\hspace{-0.6cm}\includegraphics[width=5cm]{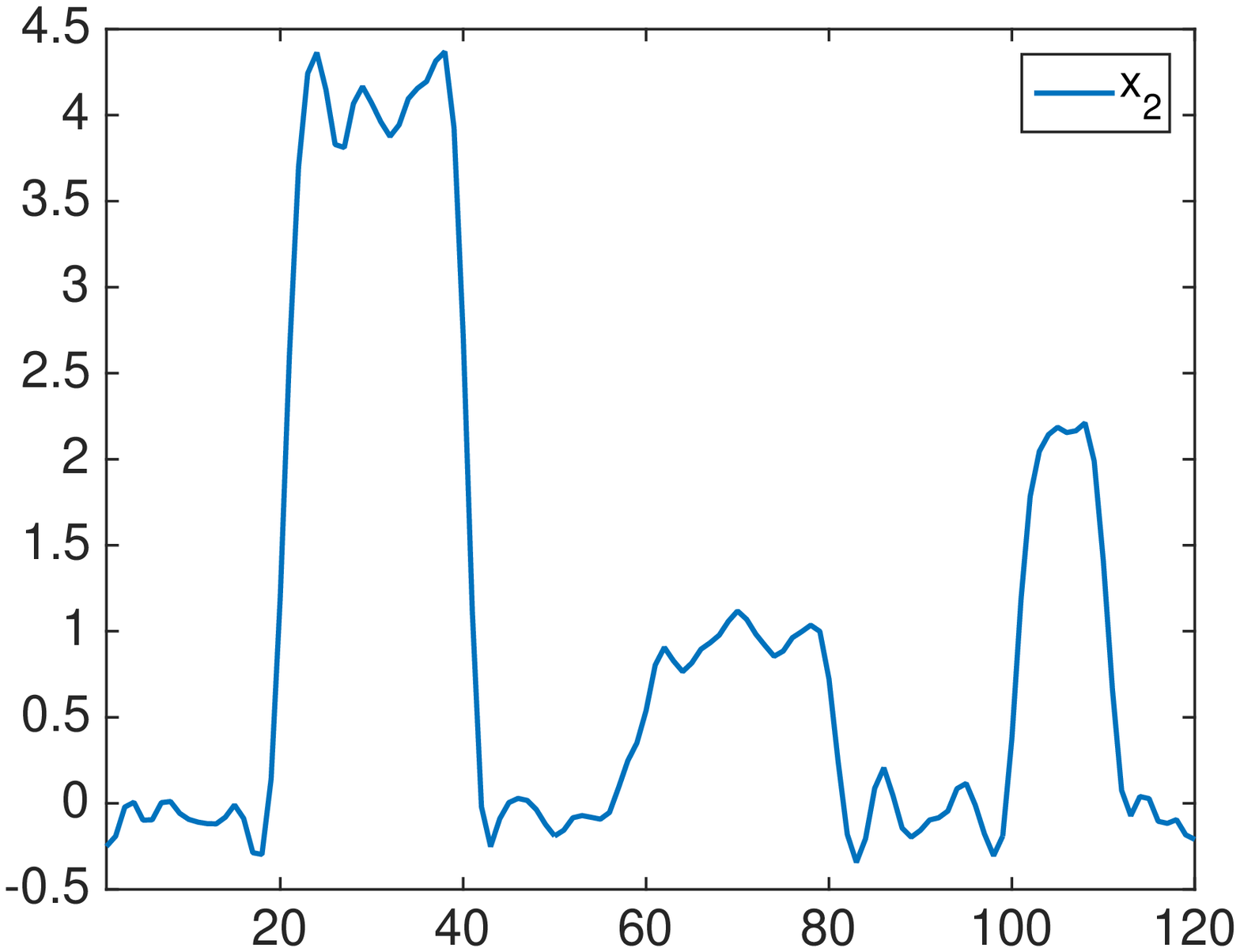}\\
\hspace{-0.6cm}{\small \textbf{(c)}} & \hspace{-0.6cm}{\small \textbf{(d)}}\\ 
\hspace{-0.6cm}\includegraphics[width=5cm]{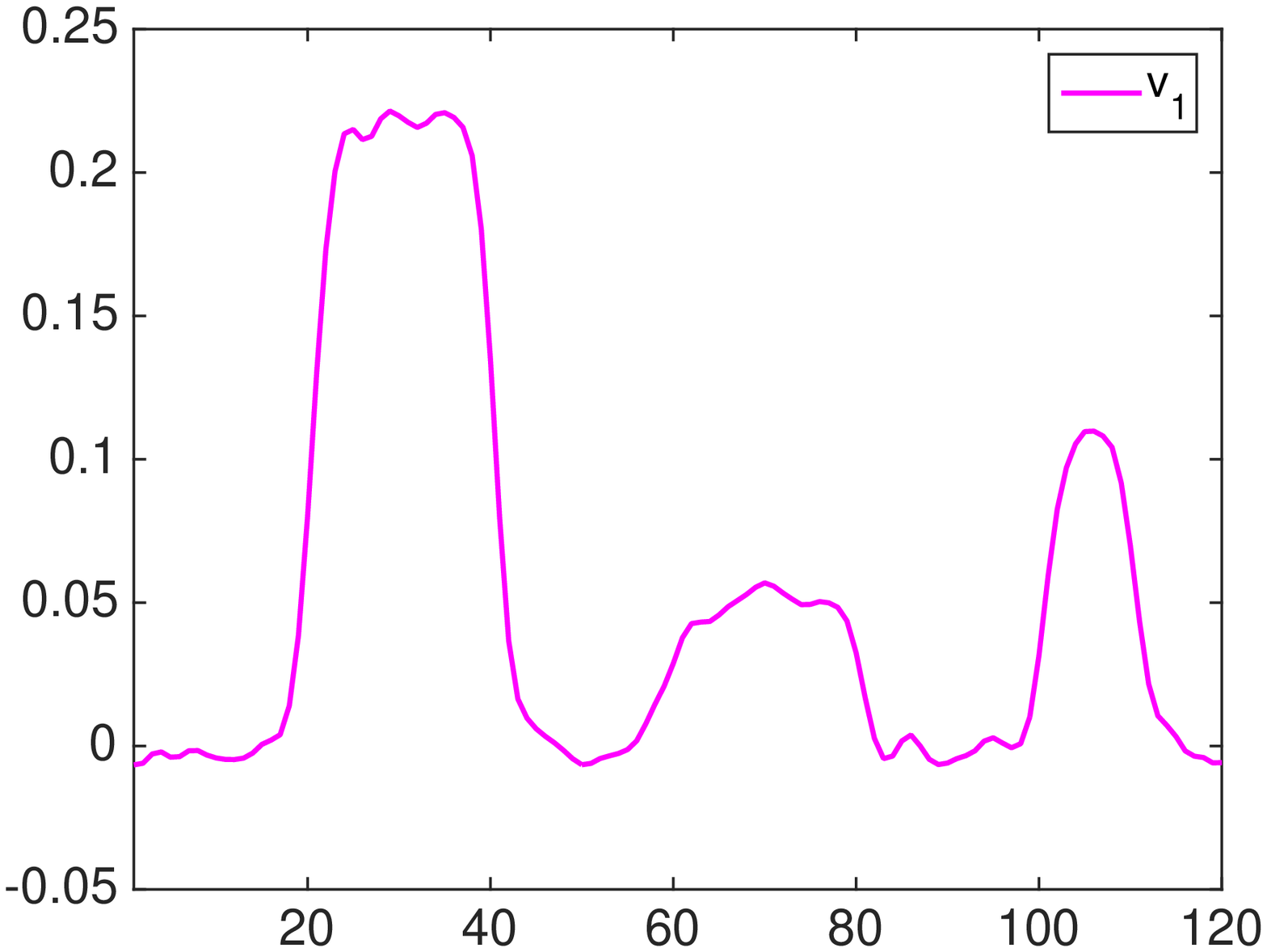} & 
\hspace{-0.6cm}\includegraphics[width=5cm]{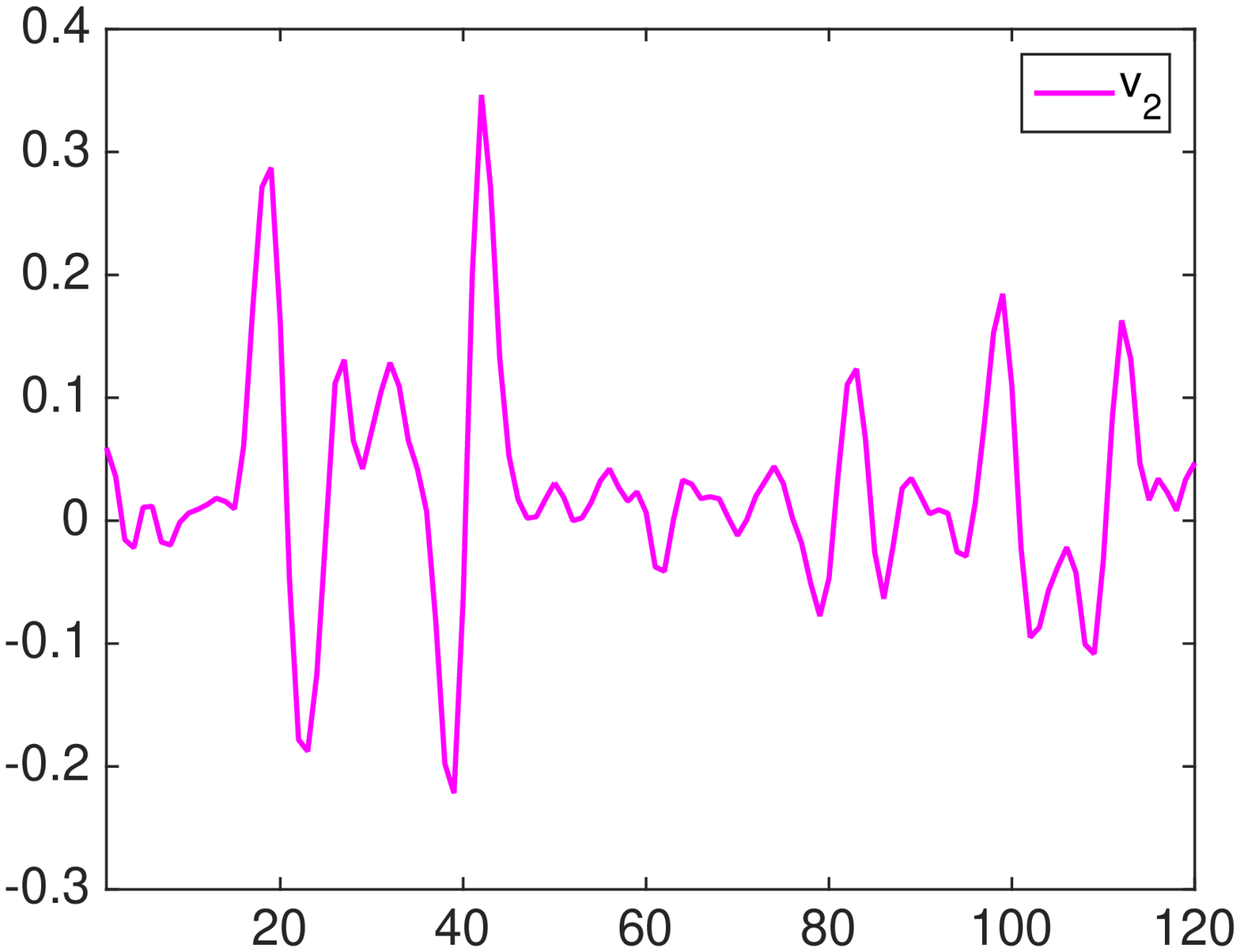}
\end{tabular}%
\caption{1D \tql image\tqr\ deblurring and denoising problem. The 1D nonnegative signal $\xex$ is displayed in frame \textbf{(a)}, along with the vector $b$ affected by Gaussian blur and Gaussian noise. Frame \textbf{(b)} displays the approximation $x_2$ obtained at the 2nd iteration of the CGLS method. Frames \textbf{(c)} and \textbf{(d)} display the {vectors $v_1$ and $v_2$}, which generate $\K_2(A^TA,A^Tb)$.}
\label{fig:example}
\end{figure}
The following considers only a version of the well-known CGLS method (cf. \cite[\S 7.4]{LS} and \cite[\S 6.7, \S 8.3]{Saad}), which can be applied to the normal equations (\ref{NE}). At the $m$th iteration, the condition
\begin{equation}\label{optRes}
x_{m} = \arg\min_{x\in x_0+\A_m}{\Phi(x)}
\end{equation}
is imposed, {with approximation subspace $\A_m=\K_m(A^TA,A^T{r}_0)$. One way of deriving CGLS is to first generate an orthonormal basis $\{v_1,\dots,v_m\}$ for the subspace 
$\K_m(A^TA,A^T{r}_0)$ for increasing values of $m$, by means of a Gram-Schmidt-like orthonormalization process. Taking $V_m=[v_1,\dots,v_m]\in\R^{N\times m}$ and considering the tridiagonal matrix $T_m = V_m^T(A^TA) V_m$,
an approximation $x_m$ for a solution of (\ref{LS}) satisfying the condition (\ref{optRes}) is obtained by solving the projected system
\begin{equation}\label{ProjSys}
T_my_m=\|r_0\|_2e_1\,,\quad e_1=[1,0,\dots,0]^T\in\R^m,
\end{equation}
and by taking $x_m=x_0+V_my_m$ \cite[Chapter 6]{Saad}}. An analogous approach can be also adopted to derive the other Krylov subspace methods listed above. 
\cref{fig:example} provides a typical example of the difficulties encountered when trying to enforce nonnegativity within Krylov methods. In this case, the CGLS method  is extremely efficient to solve an unconstrained 1D deblurring problem. Indeed, after only 2 iterations {and with $x_0=0$}, the behaviour of $\xex$ is quite accurately recovered. Imposing nonnegativity constraints would enhance even more  the quality of the reconstruction. However, once $V_2$ is computed, $x_2$ is determined by only 2 parameters (i.e., the two entries of $y_2$ in (\ref{ProjSys})), and choosing them in order have $x_2$ nonnegative would heavily modify the overall behavior of the solution. Moreover, trivially trying to project the solution $x_2$ onto the nonnegative orthant would result in a new approximation not belonging to $\K_2(A^TA,A^Tb)$: the method so obtained would rapidly stagnate, with poor approximation properties (cf. Section \ref{sect:NumExp}). 
To the best of our knowledge, only heuristic approaches have been derived so far to approximate the solution of (\ref{NNLS}) within a Krylov subspace framework.
The strategies in \cite{CLRS, Kaufman} rely on inner-outer iteration cycles. In particular, the author of \cite{Kaufman} proposes to solve {the nonlinear system in} (\ref{KKT}) with a modified CGLS method: the matrix $X$ is only updated at the beginning of each outer iteration, so that it is fixed during each inner iteration cycle. The occurrence of a restart is determined by the amount of variations in two consecutive approximations of $\xex$. The authors of \cite{CLRS} instead propose to employ the CGLS, GMRES, or RRGMRES methods during the inner iterations to solve problems like (\ref{LS}),
and restart with the nonnegatively projected last approximation of $\xex$ as soon as the discrepancy principle is satisfied. Although very efficient, this approach can only guarantee nonnegativity at each restart, and not during the inner iterations.
\cref{tab:acron} 
lists the acronyms associated with the most notable methods introduced so far.
\begin{table}
\centering
\footnotesize
\caption{Popular methods for the solution of problems (\ref{NNLS}) or (\ref{covNNLS}), related acronyms, and bibliographic references.}\label{tab:acron}
\begin{tabular}{|llcc|}
\hline
{\textbf{acronym}} & {\textbf{method}} & {\textbf{problem}} & {\textbf{reference}}\\\hline
NNSD & Projected Steepest Descent  & (\ref{NNLS}), (\ref{covNNLS}) & \cite{BerNagy13}\\ \hline
FISTA & Fast iterative shrinkage thresholding algorithms (ISTA)  & (\ref{NNLS}), (\ref{covNNLS}) & \cite{FISTA}\\ \hline
MRNSD & Modified Residual-Norm Steepest Descent & (\ref{NNLS}) & \cite{NS00}\\ \hline
PMRNSD & Preconditioned MRNSD & (\ref{NNLS}) & \cite{NS00} \\ \hline
WMRNSD & Weighted MRNSD & (\ref{covNNLS}) & \cite{BN06}\\ \hline
KWMRNSD & $k$-weighted MRNSD & (\ref{covNNLS}) & \cite{BN06}\\ \hline
ReSt NNCG & Restarted CGLS with nonnegativity at each restart & (\ref{NNLS}) & \cite{CLRS}\\ \hline
\end{tabular}
\end{table}

\subsection{{Contributions}}\label{sect:contr}

The goal of this paper is to present a reliable, efficient, and still {somewhat} heuristic new method to enforce nonnegativity within Krylov subspace methods. The new {approach} merges the ability of delivering high-quality approximations typical {for} the MRNSD methods, with a fast convergence typical {for} Krylov methods for unconstrained problems.
The example proposed in Section \ref{sect:survey} suggests that, to succeed in this task, one should modify the usual space $\K_m(A^TA, A^Tr_0)$, and its basis vectors collected in $V_m$. 

The starting point for the new method {is the nonlinear system of equations in} (\ref{KKT}). Similarly to the MRNSD method, at the $m$th step of an iterative solver, the approximation $X\simeq\dg(x_{m-1})=:X^{(m)}$ is chosen, so that the iteration-dependent linear system
\begin{equation}\label{LeftPrecLS}
X^{(m)}A^T(b-A x)=0
\end{equation}
should be solved. {The condition $x_m\geq0$ is imposed, while the condition $A^T(A x_m-b)\geq0$ is discarded during the iterations. The reason behind this choice lies in the \tql semi-convergence\tqr\ phenomenon \cite{Survey15, PCH10}, which is typical of regularizing iterative methods applied to solve least squares problems (\ref{LS}), where the matrix $A$ is of ill-determined rank, and the data vector $b$ is affected by noise. \tql Semi-convergence\tqr\ means that the solution $x_m$ approximates the solution $\xex$ of (\ref{exact}) at the beginning of the iterations, while it approaches the un-regularized and noise-dominated solution of (\ref{LS}) during the following iterations. \tql Semi-convergence\tqr\ can usually be limited if some additional regularization is imposed during the iterative process. When problem (\ref{NNLS}) is solved (or, equivalently, when conditions (\ref{KKT}) are satisfied), \tql semi-convergence\tqr\ could still occur, as a nonnegative solution minimizing $\Phi(x)$ could be heavily corrupted by noise. Therefore, the goal of the new method is to efficiently deliver a nonnegative solution $x_m$ at each iteration, and approximately satisfy the nonlinear system in (\ref{KKT}). More precisely,} the new idea is to consider a CGLS-like method {for the normal equations (\ref{NE}), devised in such a way that the variable {left} \tql preconditioner\tqr\ $X^{(m)}$ {in (\ref{LeftPrecLS})} can be handled within the iterations,} i.e., a so-called Flexible CGLS (FCGLS) method. 
The word \tql preconditioner\tqr\ has been quoted, since the goal when solving (\ref{LeftPrecLS}) is to impose regularization within the iterations: therefore, in this setting, preconditioning is not intended in a classical sense, i.e., with the aim of accelerating the convergence of an iterative method. Flexible Krylov subspace methods were originally introduced a couple of decades ago to allow an increasingly improved preconditioner at each iteration of a standard Krylov subspace method: a typical instance is when the system defining the preconditioner is solved iteratively (possibly by another Krylov subspace method) with variable tolerance on the stopping criterion (cf. \cite[Chapter 12]{Axe}, \cite[\S 9.4]{Saad}, \cite{Simoncini}, and the references therein).  The new FCGLS method derived in this paper is related to the Flexible CG (FCG) method described in \cite{FCG}: indeed, FCGLS can be regarded as FCG applied to a normal equation system.

The {specific} use of flexible Krylov subspaces for regularizing linear ill-posed problems is quite recent (cf. \cite{sparsity, ReichelFGMRES}). 
Both the approaches in \cite{sparsity, ReichelFGMRES}, and the one proposed in this paper, aim at regularizing the original problem by including new information about the solution as soon as a new iteration is computed. For this reason, flexible methods are inherently very efficient. Indeed,
when adopting a restarting strategy with a nonnegatively projected initial guess or an updated \tql preconditioner\tqr\ for the new iterations \cite{CLRS, Kaufman}, a new Krylov subspace is generated at each outer iteration.
On the contrary, the nonnegative FCGLS (NN-FCGLS) method presented in this paper (i.e., FCGLS devised to specifically deal with system (\ref{LeftPrecLS}) and secure $x_m\geq0$) relies on both flexibility and suitable restarts, in such a way that only one (flexible) Krylov subspace is generated during the iterations. 

NN-FCGLS can be also extended to handle problems affected by both Gaussian and Poisson noise, so that an approximate solution of (\ref{covNNLS}) is computed. If the covariance matrix $\covM$ is fixed, i.e., when approximation (\ref{covMfix}) is used, 
then the NN-FCGLS scheme can incorporate an additional preconditioner. If the covariance matrix $\covM$  is adaptive, i.e., when approximation (\ref{covMvar}) is used, then the NN-FCGLS scheme can incorporate an updated $\covM^{(k)}$ at the $k$th restart. In this way, by exploiting the potentialities of flexible Krylov subspaces, the newly-proposed strategies
embrace and improve the class of the MRNSD methods \cite{BN06}.

The remaining part of this paper is organized as follows: in Section \ref{sect:flexi}, the FCGLS method is derived {in a general framework}. 
In Section \ref{sect:NN}, {the} NN-FCGLS {for the approximate solution of (\ref{NNLS})} is introduced, and some its properties are discussed. In Section \ref{sect:covMvar}, two extensions of NN-FCGLS for the approximate solution of (\ref{covNNLS}) are presented. 
Section \ref{sect:NumExp} displays many numerical experiments performed with the new methods, and comparisons with the strategies reviewed in Section \ref{sect:survey}. Finally, Section \ref{sect:final} draws some conclusions and presents possible extensions.

\section{Flexible Krylov subspace methods}\label{sect:flexi}

Krylov subspace methods can be formulated as in Section \ref{sect:survey} or, alternatively, in a mathematically equivalent way that consists in explicitly updating the current solution along a set of search directions: a new search direction is added at each iteration \cite[Chapter 12]{Axe}.
More precisely, in the most general case, at the $m$th iteration of a Krylov subspace method for solving (\ref{LS}), one requires
\begin{equation}\label{updatex}
x_{m}=x_{m-1}+\sum_{j=0}^{m-1}\alpha_j^{(m-1)}d_j
\end{equation}
and computes the new search directions as
\begin{equation}\label{updated}
d_{m}=\bz_{m}+\sum_{j=0}^{m-1}\beta_j^{(m-1)}d_j\,,
\end{equation}
where $\bz_{m}$ depends on the approximation space chosen in (\ref{optRes}). This way of defining Krylov subspace methods is natural when imposing nonnegativity constraints, since one has the direct expression (\ref{updatex_simple}) for $x_m$. For CGLS-like methods applied to solve (\ref{NE}), $\bz_m$ is a vector related to the normal equation residual $A^Tr_{m}$, and the coefficients $\alpha_j^{(m-1)}$, $\beta_j^{(m-1)}$, $j=0,\dots,m-1$, are determined at each iteration by imposing the optimality condition (\ref{optRes}) and 
the orthogonality of $A d_i$, $i=0,\dots,m-1$. 
To enforce (\ref{optRes}), one takes $\partial {\Phi}(x_m)/\partial\alpha_i^{(m-1)}=0$ for $i=0,\dots,m-1$. This amounts to
\begin{equation}\label{OrthResRel}
(b-A x_m,A d_i) = 0\,,\quad i=0,\dots,m-1\,,
\end{equation}
or, equivalently, by exploiting (\ref{updatex}) and the orthogonality of $A d_i$,
\[
(r_{m-1},A d_i) = \sum_{j=0}^{m-1}\alpha_j^{(m-1)}(A d_j,A d_i)=\alpha_i^{(m-1)}(A d_i,A d_i)\,.
\]
Here and in the following, $(\cdot,\cdot)$ denotes the standard inner product on $\R^N$ or $\R^M$.
By considering (\ref{OrthResRel}) at step $m-1$, one gets that the only nonvanishing quantity on the left side of the previous equality is $(r_{m-1},A d_{m-1})$, and, therefore,
\begin{equation}\label{alpha}
\alpha_{m-1}^{(m-1)}=\frac{(r_{m-1},A d_{m-1})}{(A d_{m-1},A d_{m-1})}=:\alpha_{m-1},\quad
\alpha_{j}^{(m-1)}=0\,,\quad j=0,\dots,m-2\,.
\end{equation}
The update formula (\ref{updatex}) simplifies as
\begin{equation}\label{updatex_simple}
x_{m}=x_{m-1}+\alpha_{m-1}d_{m-1}\,,
\end{equation}
and the corresponding residual $r_m$ can be analogously updated as
\begin{equation}\label{updater}
r_{m}=r_{m-1}-\alpha_{m-1}A d_{m-1}\,.
\end{equation}
By enforcing the orthogonality of $A d_j$ in (\ref{updated}) it is immediate that
\begin{equation}\label{def_beta}
\beta_j^{(m-1)}=-\frac{(A \bz_m,A d_j)}{(A d_j,A d_j)}\,,\quad j=0,\dots,m-1\,.
\end{equation}
In the particular case of the unpreconditioned CGLS method, one takes $\bz_j=A^Tr_j{=:z_j}$. By exploiting (\ref{updated}) and (\ref{OrthResRel}), $\alpha_{m-1}$ in (\ref{alpha}) can be alternatively redefined as
\begin{equation}\label{def_alpha}
\alpha_{m-1}=\frac{(z_{m-1},d_{m-1})}{(A d_{m-1},A d_{m-1})}=\frac{(z_{m-1},\bz_{m-1})}{(A d_{m-1},A d_{m-1})}
=\frac{(z_{m-1},z_{m-1})}{(A d_{m-1},A d_{m-1})}\,;
\end{equation}
moreover, after some straightforward algebraic manipulations that mainly involve (\ref{OrthResRel}) and (\ref{updater}), relation (\ref{updated}) reduces to
\begin{equation}\label{updated_simple}
d_{m}=\bz_{m}+\beta_{m-1}d_{m-1}=z_{m}+\beta_{m-1}d_{m-1}\,,
\end{equation}
where, to keep the notation light, 
\[
\beta_{m-1}^{(m-1)}=-\frac{(A \bz_m,A d_{m-1})}{(A d_{m-1},A d_{m-1})}
=\frac{(\bz_m,z_m-z_{m-1})}{(A d_{m-1},A d_{m-1})\alpha_{m-1}}=
\frac{(z_m,z_m)}{(z_{m-1},z_{m-1})}=:\beta_{m-1}\,.
\]
The simple two-term update formula (\ref{updated_simple}) is linked to the fact that the matrix $T_m$ in (\ref{ProjSys}) is tridiagonal. A fixed symmetric positive definite left preconditioner {$L\in\R^{N\times N}$ for the system (\ref{NE})} can be efficiently incorporated into CGLS by taking $\bz_j=Lz_j=LA^T r_j$ (see \cite[Chapter 9]{Saad}).  

When considering {an iteration-dependent left preconditioner for the system (\ref{NE}), i.e., when solving a system of the form
\begin{equation}\label{LeftPrecLSL}
L^{(m)}A^TA x = L^{(m)}A^Tb\,,
\end{equation}
where the matrix $L^{(m)}\in\R^{N\times N}$ may vary at each iteration, }the short recurrence formula (\ref{updated_simple}) does not hold anymore. In theory, in these situations, the full recurrence (\ref{updated}) should be implemented. {In the following, the main steps to derive a new flexible version of CGLS, dubbed FCGLS, are outlined.} {The starting points for FCGLS are still the update formulas (\ref{updatex}) and (\ref{updated}), where 
\begin{equation}\label{varz}
\bz_m=L^{(m)}z_m=L^{(m)} A^T r_{m}
\end{equation}
in the latter. By enforcing condition (\ref{optRes}), and requiring the vectors $A d_i$, $i=0,\dots,m$, to be orthogonal, the solution can be expressed as in (\ref{updatex_simple}), with $\alpha_{m-1}$ given by (\ref{alpha}), while the new descent direction is given by (\ref{updated}), with the scalars $\beta_j^{(m-1)}$ given by (\ref{def_beta}). Note that, because of multiplication by the iteration-dependent matrix $L^{(m)}$, the full recurrence (\ref{updated}) should be implemented. Therefore, all the vectors $d_j$ and $A d_j$, $j=0,\dots,m-1$, must be stored in order to compute $d_m$ as in (\ref{updated}): this may result in a high storage cost as the number of iterations increases. The computational cost of each iteration of the FCGLS method is dominated by one matrix-vector product with $A^T$, one matrix-vector product with $A$, the update of $L^{(m)}$, and one matrix-vector product with $L^{(m)}$. More precisely, at the $m$th iteration, the residual $r_m$ must be multiplied by $A^T$ in order to compute the normal equation residual $z_m$, and the preconditioned normal equation residual $\bz_m=L^{(m)}z_m$ must be multiplied by $A$ in order to compute the coefficients $\beta_j^{(m-1)}$ and the new vectors $d_m$ and $A d_m$ (relation (\ref{updated}) premultiplied by $A$ is used for the latter).}

{To avoid high storage costs, sometimes it is appropriate to truncate the recursion (\ref{updated}) and retain at most $\hat{m}>0$ terms. If a maximum number of iterations $m_{\max}$ is assigned, the choice $\hat{m}=m_{\max}$ corresponds to the full (untruncated) recursion (\ref{updated}), while the choice $\hat{m}=1$ corresponds to the CGLS-like recurrence (\ref{updated_simple}). In \cite{FCG}, a cyclic approach for defining the truncation parameter is outlined, which basically gives rise to a \tql mixed truncation-restart\tqr\ strategy. When truncation happens, the orthogonality of all the directions $A d_j$, $j=0,\dots,m-1$ does not hold anymore; as a consequence, 
also the optimality property (\ref{optRes}) is not guaranteed anymore (see (\ref{OrthResRel}) - (\ref{updatex_simple})). However, in the FCG case, the author of \cite{FCG} claims that no deterioration of the convergence might happen, and that the biggest difference in the behavior of the truncated and untruncated versions should be expected when the extremal eigenvalues are well separated (see also \cite{Simoncini} and the references therein). In the case of matrices with ill-determined rank, the largest singular values are typically separated, while the smallest ones are clustered, so nothing can be concluded in principle. The above derivations are summarized in Algorithm 1.} 

\begin{algorithm}
\caption{Flexible CGLS (FCGLS) method}
\label{alg:FCGLS}
\begin{algorithmic}
\State Input: $A$, $b$, $L^{(0)}$, $x_0$, $\hat{m}$, $m_{\max}$. 
\State Initialize: 
$r_0 = b- A x_0$, $z_0 = A^Tr_0$, $\bz_0=L^{(0)}z_0$.
\State Take $d_0=\bz_0$, and compute $w_0 := A d_0 = A \bz_0$.
\State For $m=1,\dots,$ till a stopping criterion is satisfied OR $m=m_{\max}$:
\begin{enumerate}
\item\label{step:1} Set $\alpha_{m-1}=(r_{m-1}, w_{m-1})/(w_{m-1},w_{m-1})$.
\item\label{step:2} Update $x_{m} = x_{m-1} + \alpha_{m-1} d_{m-1}$.
\item\label{step:3} Update $L^{(m)}$.
\item\label{step:4} Update $r_{m} = r_{m-1} - \alpha_{m-1} w_{m-1}$.
\item\label{step:z1} Compute $z_{m} = A^Tr_{m}$ and $\bz_{m} = L^{(m)}z_{m}$.
\item Compute $A \bz_m$.
\item\label{step:ind1} For $j=\max\{0,m-\hat{m}\},\dots,m-1$, set 
$\beta_{j}^{(m-1)}=-(A \bz_{m},w_j)/(w_j,w_j)$.
\item\label{step:ind2} Compute $d_{m} = \bz_{m} + \sum_{j=\max\{0,m-\hat{m}\}}^{m-1}\beta_j^{(m-1)}d_j$.
\item\label{step:10} Update $w_{m}:= A d_m = A\bz_{m} + \sum_{j=\max\{0,m-\hat{m}\}}^{m-1}\beta_j^{(m-1)}w_j$.
\end{enumerate}
\end{algorithmic}
\end{algorithm}

{Algorithm \ref{alg:FCGLS} breaks down at the $m$th iteration if $\alpha_{m-1}=0$. This means that, at the $(m-1)$th iteration, a descent direction $d_{m-1}$ has been computed, which is orthogonal to the current normal equation residual $z_{m-1}$. Although a formal convergence proof for FCGLS would require additional assumptions on $A$ and $L^{(m)}$, one can claim that, if no breakdown happens in the untruncated version of Algorithm \ref{alg:FCGLS}, then $x_m$ converges monotonically to a solution of (\ref{LS}), i.e., $\|r_{m}\|_2\leq\|r_{m-1}\|_2$. This is immediate from the fact that the approximation subspaces of the solution are nested, and the optimality condition (\ref{optRes}) on the residual is imposed at each iteration.}

\section{Incorporating nonnegativity constraints}\label{sect:NN}
{In order to approximate a solution of (\ref{NNLS}), the KKT conditions are enforced, i.e., at the $m$th iteration of an iterative solver, the system (\ref{LeftPrecLS}) is considered and the constraint $x_m\geq0$ is imposed. The normal equation system  (\ref{LeftPrecLS}) can be regarded as a left-preconditioned system (\ref{LeftPrecLSL}), where the variable \tql preconditioner\tqr\ is defined at the $m$th iteration as $L^{(m)}:=X^{(m)}=\dg(x_{m-1})$. Therefore, the FCGLS method can be in principle used to solve (\ref{LeftPrecLS}). However, some modifications of the generic framework outlined in Algorithm \ref{alg:FCGLS} should be considered, which will lead to the NonNegative Flexible CGLS (NN-FCGLS) method described in Algorithm \ref{alg:NN-FCGLS}.}

{First of all, a nonnegative initial guess $x_0$ should be set; typical choices for $x_0$ are the projections of $b$ or $A^Tb$ onto the nonnegative orthant. Moreover, since a solution update of the form (\ref{updatex_simple}) is considered, one should bound the step-length  along the search directions to  guarantee nonnegativity at each iteration (this remedy is already suggested in \cite{BN06, Kaufman, NS00}). It is immediate to prove that, in the FCGLS case, the bounded step-length $\balpha_{m-1}$ is computed as follows:}
\begin{equation}\label{def_alpha_mod}
\balpha_{m-1} = \min\left(\,\alpha_{m-1}, \min\left( -x_{m-1}(d_{m-1}<0)/d_{m-1}(d_{m-1}<0) \right) \,\right),\,
\end{equation}
where $\alpha_{m-1}$ is defined as in (\ref{def_alpha}), and the MATLAB-like notations $x_{m-1}(d_{m-1}<0)$ and $d_{m-1}(d_{m-1}<0)$ mean that only the components of the vectors $x_{m-1}$ and $d_{m-1}$ corresponding to negative values of $d_{m-1}$ are evaluated, respectively. {In this way, the new step-length $\balpha_{m-1}$ is automatically determined by simultaneously imposing $x_m\geq 0$, the optimality condition (\ref{optRes}), and the orthogonality of $A d_j$, $j=0,\dots,m-1$. Therefore, if the full recurrence (\ref{updated}) is considered (i.e., if $\hat{m}=\mmax$ in Algorithm \ref{alg:FCGLS}), the NN-FCGLS residuals decrease monotonically, i.e., $\|r_{m}\|_2\leq\|r_{m-1}\|_2$. Moreover, because of the additional constraint $x_m\geq 0$, the NN-FCGLS residuals may decrease slower than the FCGLS and the standard CGLS ones (applied to solve the unconstrained least squares problem (\ref{LS})).}
The following holds:
\begin{proposition}
The scalar $\balpha_{m-1}$ defined in (\ref{def_alpha_mod}) is nonnegative.
\end{proposition}
\begin{proof}
When $\balpha_{m-1}=\alpha_{m-1}$, directly from (\ref{alpha}) it suffices to prove that $(r_{m-1},A d_{m-1})$ is nonnegative. This follows from
\begin{eqnarray*}
(r_{m-1},A d_{m-1}) &=& \left(r_{m-1},A\left(\bz_{m-1}
+\sum_{j=\max\{0, m-\hat{m}\}}^{m-2}\beta_j^{(m-2)}d_j\right)\right)\\
&=& (r_{m-1},A\bz_{m-1})+\sum_{j=\max\{0, m-\hat{m}\}}^{m-2}\beta_j^{(m-2)}(r_{m-1},A d_j)=
(r_{m-1},A\bz_{m-1})\\
&=& (A^Tr_{m-1},\bz_{m-1}) = (z_{m-1},X^{(m-1)}z_{m-1})\geq0\,,
\end{eqnarray*}
where the following are exploited: 
the property (\ref{OrthResRel}), the truncated update at step \ref{step:ind2} of Algorithm \ref{alg:FCGLS}, and the fact that the entries of the diagonal matrix $X^{(m-1)}$ are nonnegative. When $\balpha_{m-1}\neq\alpha_{m-1}$, $\balpha_{m-1}$ is nonnegative by definition.
\end{proof}

It must be remarked that the FCGLS method with $\alpha_{m-1}$ in (\ref{def_alpha})
replaced by $\balpha_{m-1}$ is very prone to stagnation. Indeed, $\balpha_{m-1}=0$ as soon as $[x_{m-1}]_i=0$ for some $i=1,\dots,N$ such that $[d_{m-1}]_i<0$. At this point, $\balpha_n=0$ for all $n\geq m$, so that no updates of $x_n$ happen, and Algorithm \ref{alg:FCGLS} with the choice (\ref{def_alpha_mod}) breaks down. The same is not true for the class of the MRNSD methods since, if $[x_{m-1}]_i=0$, then $[d_{m-1}]_i=0$, cf. (\ref{MRNSD}). In order to overcome stagnation, NN-FCGLS relies on suitable restarts of FCGLS: a restart happens as soon as a maximum number of inner iterations $\mmax^{\text{in}}$ is performed, or $\balpha_m=0$. The iterations are terminated as soon as a stopping criterion is satisfied or, alternatively, when a maximum number of outer iterations $\mmaxout$ is performed.
If a good estimate of the norm of the noise $\|\eta\|_2$ is known, a typical stopping criterion is the discrepancy principle (see \cite[Chapter 5]{PCH10} and \cite{CLRS}), which prescribes to terminate the iterations when $\|r_m\|_2$ drops below $\|\eta\|_2$. {If $\|\eta\|_2$ is not available, then one can monitor the relative change in the residual norms between two successive iterations, 
and stop when it drops below a prescribed tolerance. In addition to preventing stagnation, the restarting strategy of NN-FCGLS is beneficial to reduce the storage requirements of FCGLS, assuming that the untruncated update (\ref{updated}) of $d_m$ is considered, and that $\mmaxin$ is low. Notation-wise, when considering NN-FCGLS, it is appropriate to denote some of the vectors appearing in Algorithm \ref{alg:FCGLS} by double indices: the lower index counts the number of inner iterations, while the upper index counts the number of outer iterations.}
\begin{algorithm}
\caption{NonNegative FCGLS (NN-FCGLS) method}
\label{alg:NN-FCGLS}
\begin{algorithmic}
\State Input: $A$, $b$, $x_0^0\geq 0$, $\hat{m}$, $\mmaxin$, $\mmaxout$.\smallskip
\State For $k=1,\dots,$ till a stopping criterion is satisfied OR $k=\mmaxout$:\medskip
\begin{itemize}
\item Take $L^{(0)}=X^{(0)}=\dg(x_0^{k-1})$,
\item[]$r_0^{k-1} = b- A x_0^{k-1}$, $\bz_0^{k-1}=L^{(0)}A^Tr_0^{k-1}$, $d_0^{k-1}=\bz_0^{k-1}$, 
and $w_0^{k-1} = A \bz_0^{k-1}$.\smallskip
\item\label{step:inner} For $m=1,\dots$ till $\balpha_{m-1}=0$, OR $m=\mmaxin$, OR a stopping criterion is satisfied:\smallskip
\begin{itemize}
\item[] Run steps \ref{step:1} -- \ref{step:10} of Algorithm \ref{alg:FCGLS}. In particular:
\begin{itemize}
\item at step \ref{step:2} compute $x_m^{k-1}$, with $\alpha_{m-1}$ replaced by $\balpha_{m-1}$, computed as in (\ref{def_alpha_mod});
\item at step \ref{step:3} take $L^{(m)}=X^{(m)}=\dg(x_m^{k-1})$;
\item at steps \ref{step:4}--\ref{step:10} compute $r_m^{k-1}$, $\bz_m^{k-1}$, $d_m^{k-1}$, 
and $w_m^{k-1}$.\smallskip
\end{itemize}
\end{itemize}
\item  Let $n_k$ be the stopping iteration. Take $x_0^{k}=x_{n_k}^{k-1}$.
\end{itemize}
\end{algorithmic}
\end{algorithm}
Some properties of Algorithm \ref{alg:NN-FCGLS} are derived in the following two results. 
\begin{proposition}\label{prop:break}
{If, at the beginning of the $k$th outer iteration, the $i$th component of $d_0^{k-1}$ is not zero, then the $i$th component of $x_0^{k-1}$ is not zero.} 
\end{proposition}
\begin{proof}
One equivalently proves that $[x_0^{k-1}]_i=0$ implies $[d_0^{k-1}]_i=0$. This follows immediately from the definition of $d_0^{k-1}$ given in Algorithm \ref{alg:NN-FCGLS}:
\[
\left[d_0^{k-1}\right]_i=\left[x_0^{k-1}\right]_i\left[A^T(b - A x_0^{k-1})\right]_i\,.
\]
\end{proof}
{Note that the above proposition might not be true for $d_{m}^{k-1}$ and $x_m^{k-1}$, with $m>0$. Indeed, it can happen that $[x_m^{k-1}]_i=0$ and 
\begin{eqnarray*}
\left[d_m^{k-1}\right]_i&=&\left[x_m^{k-1}\right]_i\left[A^T(b - A x_m^{k-1})\right]_i
+\sum_{j=\max\{0,m-\hat{m}\}}^{m-1}\beta_j^{(m-1)}\left[d_j^{k-1}\right]_i\\
&=&0+\sum_{j=\max\{0,m-\hat{m}\}}^{m-1}\beta_j^{(m-1)}\left[d_j^{k-1}\right]_i\neq 0\,.
\end{eqnarray*}
If, in particular, $[d_m^{k-1}]_i<0$, then the quantity $\balpha_{m}$ in (\ref{def_alpha_mod}) would be zero, and a restart would happen in the NN-FCGLS method. \cref{prop:break} is important to assure that Algorithm \ref{alg:NN-FCGLS} does not stagnate, i.e., at least one inner iteration can be computed during each outer iteration cycle, unless $\alpha_0$ in (\ref{alpha}) is zero. If $\alpha_0=0$, then $X^{(0)}A^T(b-A x_0^{k})=0$, and $x_0^{k}$ is a solution of the nonlinear system in (\ref{KKT}).}
The following property also holds:
\begin{proposition}\label{prop:zero}
Assume that $[x_0^{k-1}]_i=0$ for some $i=1,\dots,N$, and that $n_k$ iterations of FCGLS are performed at the $k$th outer cycle of Algorithm \ref{alg:NN-FCGLS}. Then $[x_m^{k-1}]_i=0$ for all $m=1,\dots,n_k$.
\end{proposition}
\begin{proof}
By induction. If $[x_0^{k-1}]_i=0$ for some $i=1,\dots,N$, then
\[
[x_1^{k-1}]_i = [x_0^{k-1}]_i+\balpha_0[d_0^{k-1}]_i=[x_0^{k-1}]_i+\balpha_0[x_0^{k-1}]_i[z_0^{k-1}]_i=0\,.
\]
In particular, $[d_0^{k-1}]_i=0$. Now assume that $[x_{2}^{k-1}]_i=0,\dots,[x_{\ell}^{k-1}]_i=0$ for $\ell<n_k$. This implies that $[d_{2}^{k-1}]_i=0,\dots,[d_{\ell-1}^{k-1}]_i=0$, since
\[
[d_{j}^{k-1}]_i=\balpha_{j}^{-1}\left([x_{j+1}^{k-1}]_i-[x_{j}^{k-1}]_i\right)=0\quad\mbox{and}\quad\balpha_{j}\neq0\quad\mbox{for}\quad 
j=0,\dots,\ell-1\,.
\]
Proving that $[x_{\ell+1}^{k-1}]_i=0$ is immediate, since 
\[
[x_{\ell+1}^{k-1}]_i = [x_{\ell}^{k-1}]_i+\balpha_{\ell}[d_{\ell}^{k-1}]_i=0+\balpha_\ell\left([x_{\ell}^{k-1}]_i[z_{\ell}^{k-1}]_i
+\sum_{j=\max\{0,\ell-\hat{m}\}}^{\ell-1}\beta_{j}^{(\ell-1)}[d_{j}^{k-1}]_i\right)=0\,.
\]
\end{proof}
The above result can be easily extended across the outer iterations, since $x_0^k=x_{n_k}^{k-1}$. Moreover, a similar property holds for the class of the MRNSD methods (this follows immediately from the update formula (\ref{MRNSD})), and it is important to guarantee that no oscillations occur around the newly-recovered zero components.

Similarly to Algorithm \ref{alg:FCGLS}, the computational cost of each iteration of Algorithm \ref{alg:NN-FCGLS} is dominated by a matrix-vector product with $A$, and a matrix-vector product with $A^T$; indeed, in the NN-FCGLS case, the cost of updating $L^{(m)}=X^{(m)}$, and computing a matrix-vector product with $X^{(m)}$ is negligible. Therefore, the cost of one iteration of NN-FCGLS is comparable to the cost of one iteration of CGLS, gradient projection (FISTA), and MRNSD. 

\section{Incorporating a covariance preconditioner}\label{sect:covMvar}

As explained in Section \ref{sect:survey}, when considering problems affected by both Gaussian and Poisson noise one has to deal with formulation (\ref{covNNLS}).
The KKT conditions associated with the constrained problem (\ref{covNNLS}) are still leveraged; similarly to (\ref{KKT}), they can be expressed as:
\begin{equation}\label{covKKT}
X(\covM^{-1/2}A)^T\covM^{-1/2}(Ax - \bbeta)=0\,,\;\mbox{where}\; X=\dg(x)\,,\; x\geq0\,,\; A^T\covM^{-1}(Ax-\bbeta)\geq 0\,.
\end{equation}
Analogously to the Gaussian noise case, the last condition is discarded, while at the $m$th step of an iterative solver for the nonlinear system in (\ref{covKKT}), the approximation $X\simeq X^{(m)}=:\dg(x_{m-1})$ is chosen, so that the iteration-dependent linear system
\begin{equation}\label{covLeftPrecLS}
X^{(m)}A^T\covM^{-1}(Ax - \bbeta)=0\,,\quad\mbox{with}\quad x_m\geq 0
\end{equation}
should be approximately solved. The linear system in (\ref{covLeftPrecLS}) has two preconditioners: while $X^{(m)}$ is updated at each iteration, at this stage $\covM$ is assumed to be fixed (and, therefore, approximation (\ref{covMfix}) is used). Because of the presence of $X^{(m)}$ and the constraint \linebreak[4]$x_m\geq0$, an iterative scheme such as NN-FCGLS must be used to approximate the solution of (\ref{covLeftPrecLS}). However, the NN-FCGLS as described in Algorithm \ref{alg:NN-FCGLS} should be reviewed in order to account for $\covM$, and the following derivations will lead to the Covariance-Preconditioned NN-FCGLS (CP-NN-FCGLS) method. The starting points for CP-NN-FCGLS are still the generic update formulas (\ref{updatex}) and (\ref{updated}), where the coefficients $\alpha_j^{(m-1)}$ and $\beta_j^{(m-1)}$, $j=0,\dots,m-1$ are set by imposing an optimality condition similar to (\ref{optRes}), i.e., $\partial\Phi_C(x_m)/\partial\alpha_i^{(m-1)}=0$, and by imposing the $\covM^{-1/2}Ad_i$ to be orthogonal, $i=0,\dots,m-1$. In particular, relation
\[
(\covM^{-1/2} r_m,\covM^{-1/2}Ad_i)=0,\quad i = 0,\dots,m-1
\]
holds and, with the same reasoning used in Section \ref{sect:flexi}, 
\[
\alpha_{m-1}^{(m-1)}=\frac{(\covM^{-1/2}r_{m-1},\covM^{-1/2}Ad_{m-1})}{(\covM^{-1/2}Ad_{m-1},\covM^{-1/2}Ad_{m-1})}=:\alpha_{m-1},\quad\alpha_j^{(m-1)}=0,\quad j=0,\dots,m-2\,,
\]
so that a short formula like (\ref{updatex_simple}) can be used to update the solution.  Concerning the vector $d_m$, by exploiting the orthogonality of $\covM^{-1/2}Ad_i$, $i=0,\dots,m-1$, the coefficients in (\ref{updated}) are computed as follows
\[
\beta_j^{(m-1)}=-\frac{(\covM^{-1/2}A\bz_m,\covM^{-1/2}Ad_j)}{(\covM^{-1/2}Ad_{j},\covM^{-1/2}Ad_{j})}\,,\quad
j=0,\dots,m-1\,,
\]
where $\bz_m=X^{(m)}A^T\covM^{-1}r_m=X^{(m)}A^T\covM^{-1}(\bbeta-Ax_m)$. In order to guarantee $x_m\geq0$ at the $m$th iteration, a bound analogous to (\ref{def_alpha_mod}) should be imposed and, similarly to NN-FCGLS, this could lead to stagnation. To overcome stagnation, an inner-outer iteration strategy is devised, which is based on restarts of the underlying FCGLS method. With the goal of  giving an alternative approximation for the covariance matrix $\covM$ defined in (\ref{GaussPoissPb}), when the $k$th restart happens one can choose to update the covariance in the following way:
\begin{equation}\label{covMres}
\covM^{(k)}=\dg(A x_{n_k}^{k-1}+\beta\bones+\sigma^2\bones)\,,
\end{equation}
where $x_{n_k}^{k-1}$ is the last approximation computed at the $(k-1)$th iteration cycle. In this way, a restart-dependent covariance matrix is defined, which is fixed during each inner iteration cycle. The previous derivations are summarized in Algorithm \ref{alg:CPNN-FCGLS}.
\begin{algorithm}
\caption{Covariance-Preconditioned NN-FCGLS (CP-NN-FCGLS) method}
\label{alg:CPNN-FCGLS}
\begin{algorithmic}
\State Input: $A$, $b$, $x_0^0\geq 0$, $\hat{m}$, $\mmaxin$, $\mmaxout$.\smallskip
\State For $k=1,\dots,$ till a stopping criterion is satisfied OR $k=\mmaxout$:\medskip
\begin{itemize}
\item Take $X^{(0)}=\dg(x_0^{k-1})$, and $\bar{\covM}=\covM$ as in (\ref{covMfix}) OR $\bar{\covM}=\covM^{(k)}$ as in (\ref{covMres});
\item[]$r_0^{k-1} = b- A x_0^{k-1}$, $\br_0^{k-1} = \bar{\covM}^{-1/2}r_0^{k-1}$, 
$\bz_0^{k-1}=X^{(0)} A^T\bar{\covM}^{-1/2}\br_0^{k-1}$, 
\item[]$d_0^{k-1}=\bz_0^{k-1}$, $w_0^{k-1} = A \bz_0^{k-1}$, and $\bar{w}_0^{k-1} = \bar{\covM}^{-1/2}w_0^{k-1}$.\smallskip
\item For $m=1,\dots$ till $\balpha_{m-1}=0$, OR $m=\mmaxin$, OR a stopping criterion is satisfied:\smallskip
\begin{itemize}
\item Set $\alpha_{m-1}=(\br_{m-1}^{k-1}, \bar{w}_{m-1}^{k-1})/(\bar{w}_{m-1}^{k-1},\bar{w}_{m-1}^{k-1})$ and take $\balpha_{m-1}$ as in (\ref{def_alpha_mod}).
\item Update $x_{m}^{k-1} = x_{m-1}^{k-1} + \balpha_{m-1} d_{m-1}^{k-1}$.
\item Update $X^{(m)}$.
\item Update $r_{m}^{k-1} = r_{m-1}^{k-1} - \balpha_{m-1} w_{m-1}^{k-1}$ and 
$\br_{m}^{k-1}=\bar{\covM}^{-1/2}r_m^{k-1}$.
\item Compute $\bz_{m}^{k-1} = X^{(m)}A^T\bar{\covM}^{-1/2}\br_{m}^{k-1}$.
\item Compute $A \bz_m^{k-1}$.
\item For $j=\max\{0,m-\hat{m}\},\dots,m-1$,
\item[]set $\beta_{j}^{(m-1)}=-(\bar{\covM}^{-1/2}A \bz_{m}^{k-1},\bar{w}_j^{k-1})/(\bar{w}_j^{k-1},\bar{w}_j^{k-1})$.
\item Compute $d_{m}^{k-1} = \bz_{m}^{k-1} + \sum_{j=\max\{0,m-\hat{m}\}}^{m-1}\beta_j^{(m-1)}d_j^{k-1}$.
\item Update $w_{m}^{k-1}= A d_m^{k-1} = A\bz_{m}^{k-1} + \sum_{j=\max\{0,m-\hat{m}\}}^{m-1}\beta_j^{(m-1)}w_j^{k-1}$,
\item[]and $\bar{w}_m^{k-1}=\bar{\covM}^{-1/2}w_m^{k-1}$
\end{itemize}
\item  Let $n_k$ be the stopping iteration. Take $x_0^{k}=x_{n_k}^{k-1}$.
\end{itemize}
\end{algorithmic}
\end{algorithm}
Like Algorithm \ref{alg:NN-FCGLS}, the computational cost of one iteration of Algorithm \ref{alg:CPNN-FCGLS} is dominated by one matrix-vector products with $A$ and one matrix-vector product with $A^T$; the cost of updating $X^{(m)}$, $\covM^{(k)}$, and $(\covM^{(k)})^{-1/2}$, and performing matrix-vector multiplications with them, is negligible.

\section{Numerical experiments}\label{sect:NumExp}

This section proposes three examples concerned with imaging problems. The first and the second ones are realistic astronomical image restoration test problems. The third one is an image reconstruction test problem that simulates an acquisition by computerized tomography with parallel beams. All the tests are performed running MATLAB R2015a on a single processor 2.2 GHz Intel Core i7. The image restoration test data are available within the toolbox \cite{RestTools}. In both Examples 1 and 2 the exact and perturbed images are of size $256\times 256$ pixels, and the blurring matrix $A$ is of size $65536\times 65536$. $A$ is not available explicitly but, using the software in \cite{RestTools}, matrix-vector products with $A$ and $A^T$ are computed recurring to the point spread function, which defines the blur. The tomography test problem is available in \cite{AIRT}, as \texttt{paralleltomo}. In Example 3 the Shepp-Logan phantom of size $256\times 256$ pixels is used as exact image, and the parameters modeling the acquisition process are the angles where the sources are located, the number of rays for each angle, and the distance between the first and the last ray. The sparse sensing matrix $A$ is of size ${M\times 65536}$, where $M$ is varied by considering different acquisition parameters.

The new NN-FCGLS method is compared with 
other state-of-the-art solvers for the nonnegatively constrained linear least squares problems (\ref{NNLS}) or (\ref{covNNLS}): almost all of them are among the ones surveyed in Section \ref{sect:survey}, whose acronyms are reported in \cref{tab:acron}. In particular, two versions of FISTA are considered: the basic one introduced in \cite{FISTA}, without backtracking (for the choice of the stepsize), and the monotonic version of FISTA (MFISTA), with backtracking, described in \cite{MFISTA}. Indeed, the performance of (M)FISTA is very much dependent on the stepsize $t$. According to the theory, one should take $t=1/(2\sigma_1^2)$, where $\sigma_1$ is the largest singular value of $A$; in practice, $t$ might be extremely expensive to compute: this is the case for all the test problems considered in this section. When testing FISTA, an approximation of $\sigma_1$ is obtained by running a few iterations of the Golub-Kahan bidiagonalization algorithm \cite{GK65}: in the following examples, 5 iterations are considered, so that the computational cost of this process is dominated by 5 matrix-vector products with $A$, and 5 matrix-vector products with $A^T$. When testing MFISTA, sometimes the notation MFISTA($1/t$) is used to emphasize the chosen stepsize; MFISTA stands for MFISTA($2\sigma_1^2$).

For each of the following examples, a table reporting the behavior of the different solvers is displayed: the second column (labeled \tql rel.error\tqr) reports the minimum relative error
\begin{equation}\label{RelErr}
\frac{\|\xex-x_{m}\|_2}{\|\xex\|_2}
\end{equation}
attained by each method, 
where the iteration number $m$ is displayed in the third column (labeled \tql iterations\tqr). The fourth column (labeled \tql tot.time\tqr) reports the total time to perform $m$ iterations, while the last column (labeled \tql av.time\tqr) shows the average time per iteration. The time is measured in seconds. All the values are averages over 10 runs of each test problem, with different noise realizations. For each test problem, the graphs show the error history, i.e., the values of the relative errors (\ref{RelErr}) at each iteration versus the number of iterations. Concerning NN-FCGLS, the two stopping criteria
\begin{equation}\label{stopC}
\frac{\left\vert\|r_{m-1}\|_2-\|r_{m}\|_2\right\vert}{\|r_{m-1}\|_2}<\tau\,,
\quad\mbox{or}\quad \frac{\|r_{m-1}\|_2}{\|b\|_2}<\theta\weps
\end{equation}
are taken into account. The first one monitors the stabilization of the residual norms, i.e., the iterative process terminates as soon as the relative difference between the residual norm of two consecutive iterates drops below a prescribed tolerance $\tau>0$; note that the absolute value must be considered when the recurrences (\ref{updated}) are truncated, as the property $\|r_{m}\|_2\leq\|r_{m-1}\|_2$ is not guaranteed anymore. The second one is the well-known discrepancy principle that, in the Gaussian noise case, can be applied only if a good estimate of the norm of the noise $\|\eta\|_2$ is known; the scalar $\theta$ is a safety factor ($\theta=1.01$ is set in the following examples), while the scalar $\weps$ is the noise level, defined as $\weps=\|\eta\|_2/\|b^{ex}\|_2$. In order to produce the graphs of the error history, NN-FCGLS runs once for each stopping criterion, and special markers are used to emphasize the iterations satisfying each stopping criterion. Note that, in the following experiments, some additional iterations are typically performed after the stoping criterion is satisfied, in order to assess the stability of the method. Special markers also highlight the restarting iterations (recall that a restart happens as soon as a maximum number of inner iterations $\mmaxin$ is performed, or $\bar{\alpha}_m$ in (\ref{def_alpha_mod}) is zero).

\textbf{Example 1.} The first experiment is concerned with the deblurring and denoising of the so-called \texttt{star\_cluster} test image of size $256\times 256$ pixels \cite{RestTools}. The matrix $A$ models a spatially variant blur, consisting of 25 locally spatially invariant point spread functions. Gaussian noise $\eta$ of level $\weps=10^{-2}$ 
is added to the blurred image $b^{ex}$. The exact image $\xex$ and the perturbed data $b$ are among the ones displayed in \cref{fig:Img_StClust}. The parameter $\mmaxin$ of Algorithm \ref{alg:NN-FCGLS} is set to 20, and the initial guess $x_0^0$ is the projection of $b$ onto the nonnegative orthant, while the parameter $\tau$ appearing in (\ref{stopC}) is set to $10^{-4}$; the untruncated version of NN-FCGLS ($\hat{m}=\mmaxin$) is implemented.
For this problem, also a \tql naive\tqr\ version of a \tql nonnegative CGLS\tqr\ method (dubbed \tql naive NNCG\tqr) is considered. Naive NNCG is obtained by simply projecting the approximation (\ref{updatex_simple}) onto the nonnegative orthant at each iteration: recalling the remarks in Section \ref{sect:survey}, the success of this strategy is not guaranteed. 
\cref{tab:Img_StClust} summarizes the performance of the different solvers, which are stopped after 400 (total) iterations. The initial guess for all the solvers is the projection of $b$ onto the nonnegative orthant.
\begin{table}
\footnotesize
\centering
\caption{\emph{\texttt{star\_cluster}} test problem, with $\widetilde{\eps}=10^{-2}$. Average values over $10$ runs of the test problem, with different noise realizations.}\label{tab:Img_StClust}
\begin{tabular}{|lcccc|}
\hline
 & {\textbf{rel.error}} & {\textbf{iterations}} & {\textbf{tot.time}} & {\textbf{av.time}}\\\hline
{\textbf{NN-FCGLS}} & 2.8132e-03 & 248.67 & 62.56 & 0.25\\ \hline
{\textbf{ReSt NNCG}} & 5.3699e-03 & 261.00 & 113.51 & 0.43\\ \hline
{\textbf{FISTA}} & 9.1283e-02 & 72.00 & 42.06 & 0.58\\ \hline
{\textbf{MFISTA}} & 3.2803e-03 & 400.00 & 216.11 & 0.54\\ \hline
{\textbf{MFISTA(0.2)}} & 3.2445e-03 & 400.00 & 194.78 & 0.49\\ \hline
{\textbf{MFISTA(5)}} & 4.2834e-03 & 400.00 & 185.22 & 0.46\\ \hline
{\textbf{MRNSD}} & 1.9889e-02 & 400.00 & 91.11 & 0.23\\ \hline
{\textbf{NNSD}} & 8.3206e-02 & 400.00 & 91.59 &  0.23\\ \hline
{\textbf{naive NNCG}} & 1.4028e-01 & 400.00 & 105.02 & 0.26 \\ \hline
\end{tabular}
\end{table}
\begin{figure}[htbp]
\centering
\includegraphics[width=14.0cm]{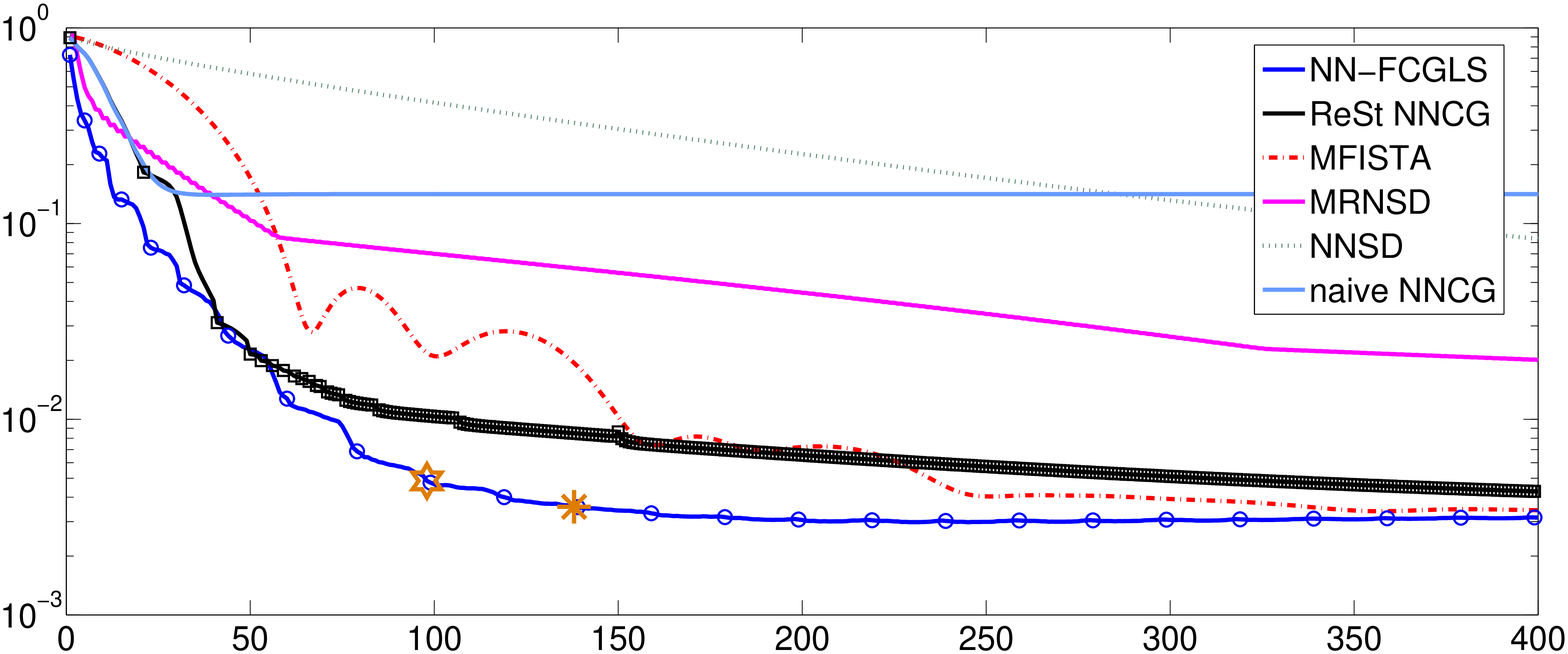}
\caption{\emph{\texttt{star\_cluster}} test problem, with $\weps=10^{-2}$. History of the relative errors. For the NN-FCGLS and ReSt CG methods, a small marker is used to emphasize the iterations where restarts happen. The big markers for NN-FCGLS highlight the stopping iterations.}
\label{fig:RelErr_StClust}
\end{figure}
\begin{figure}[htbp]
\centering
\begin{tabular}{cc}
\hspace{-0.1cm}{\small \textbf{(a)}} & {\small \textbf{(b)}}\\ 
\hspace{-0.1cm}\includegraphics[width=6.2cm]{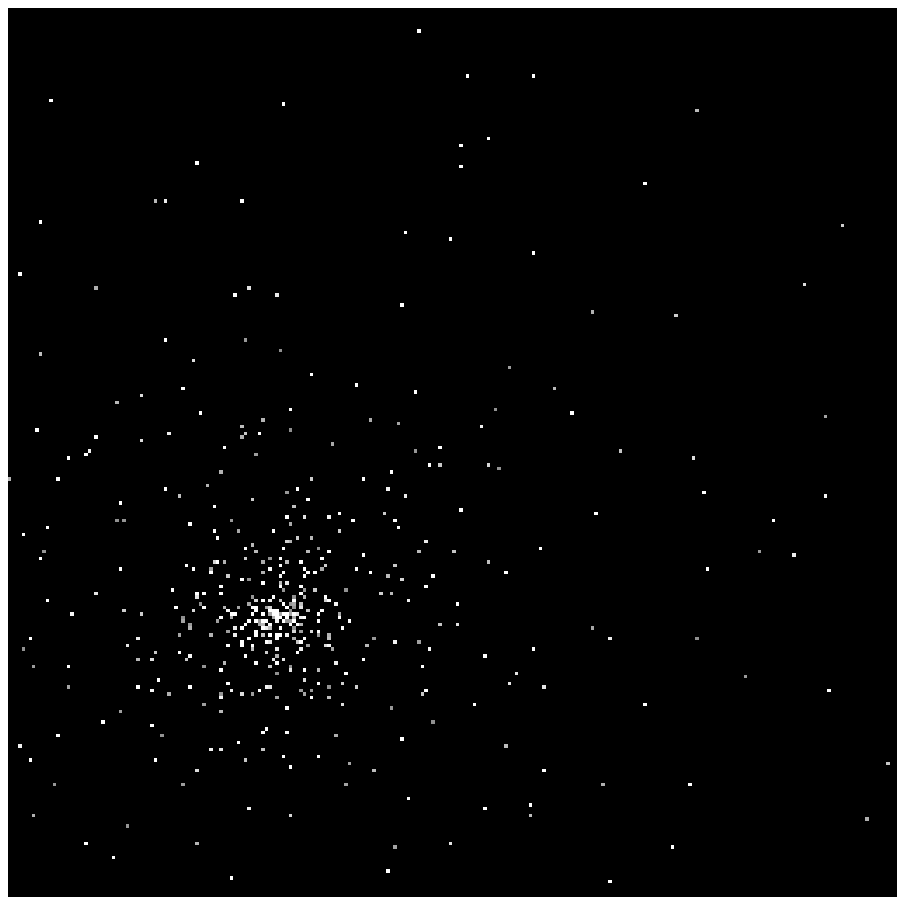} & 
\includegraphics[width=6.2cm]{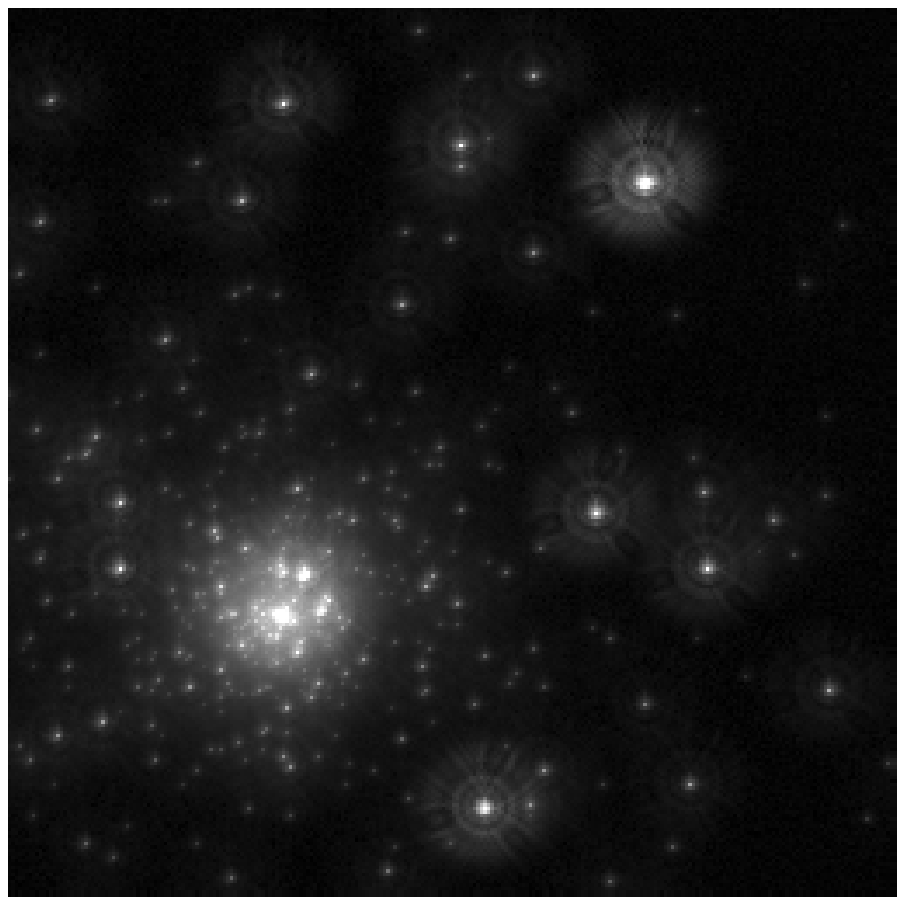}\\
\hspace{-0.1cm}{\small \textbf{(c)}} & {\small \textbf{(d)}}\\ 
\hspace{-0.1cm}\includegraphics[width=6.2cm]{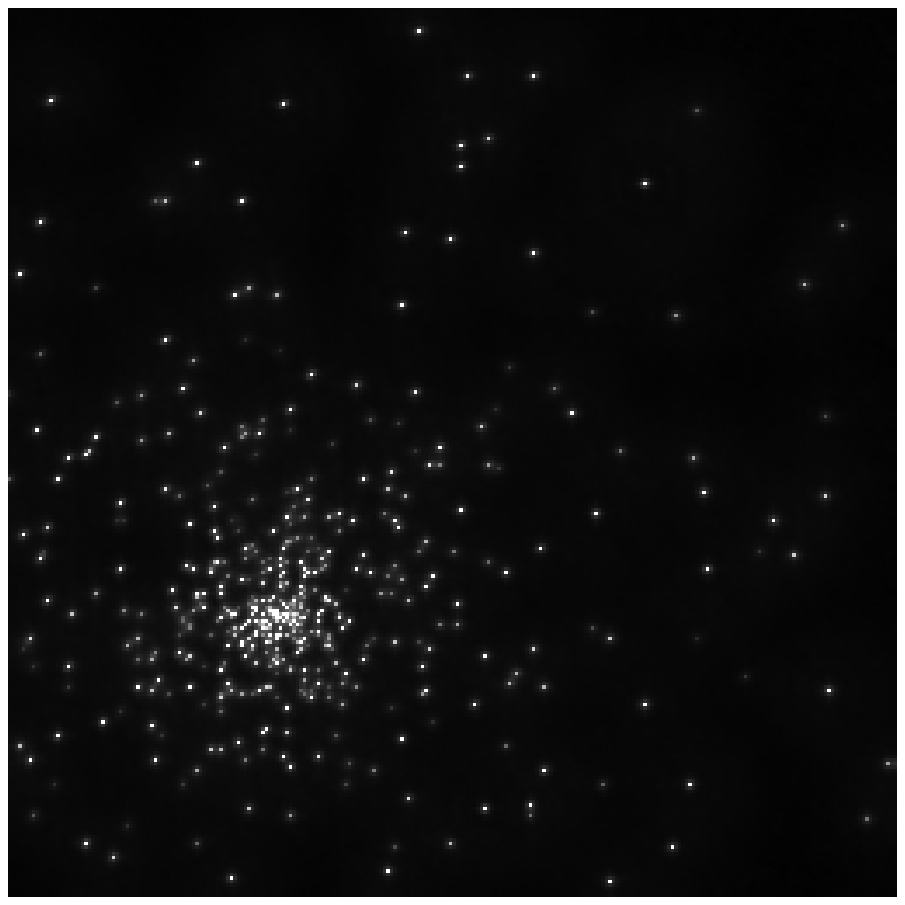} & 
\includegraphics[width=6.2cm]{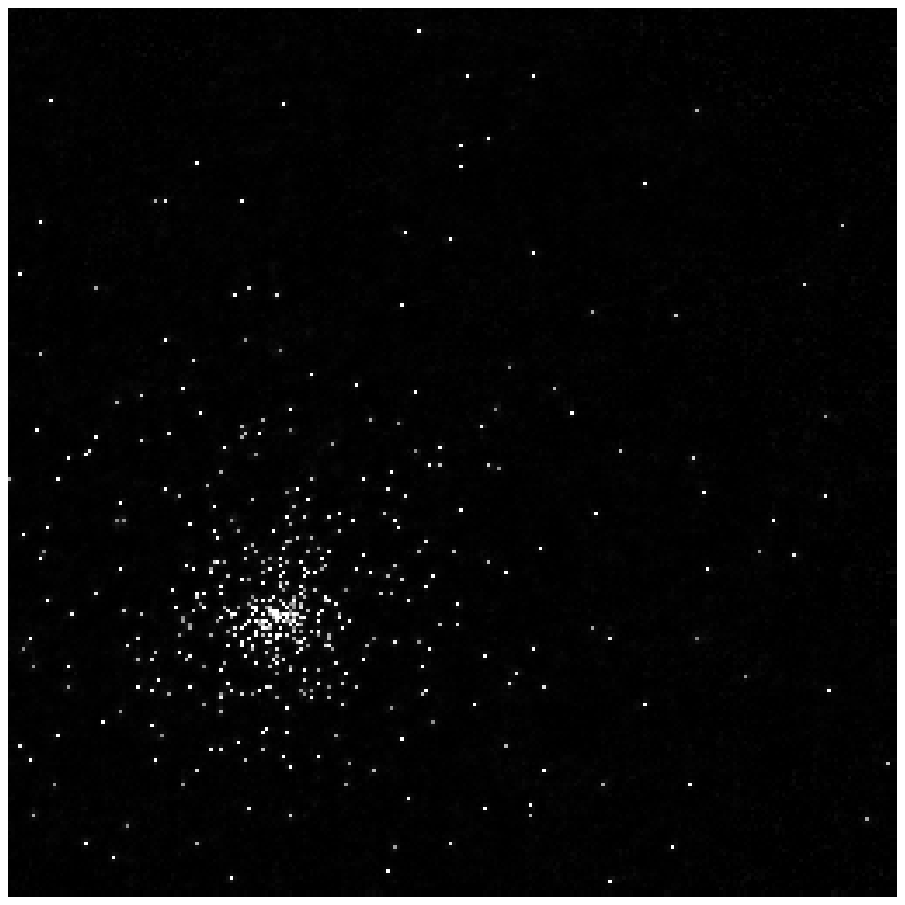}\\
\hspace{-0.1cm}{\small \textbf{(e)}} & {\small \textbf{(f)}}\\ 
\hspace{-0.1cm}\includegraphics[width=6.2cm]{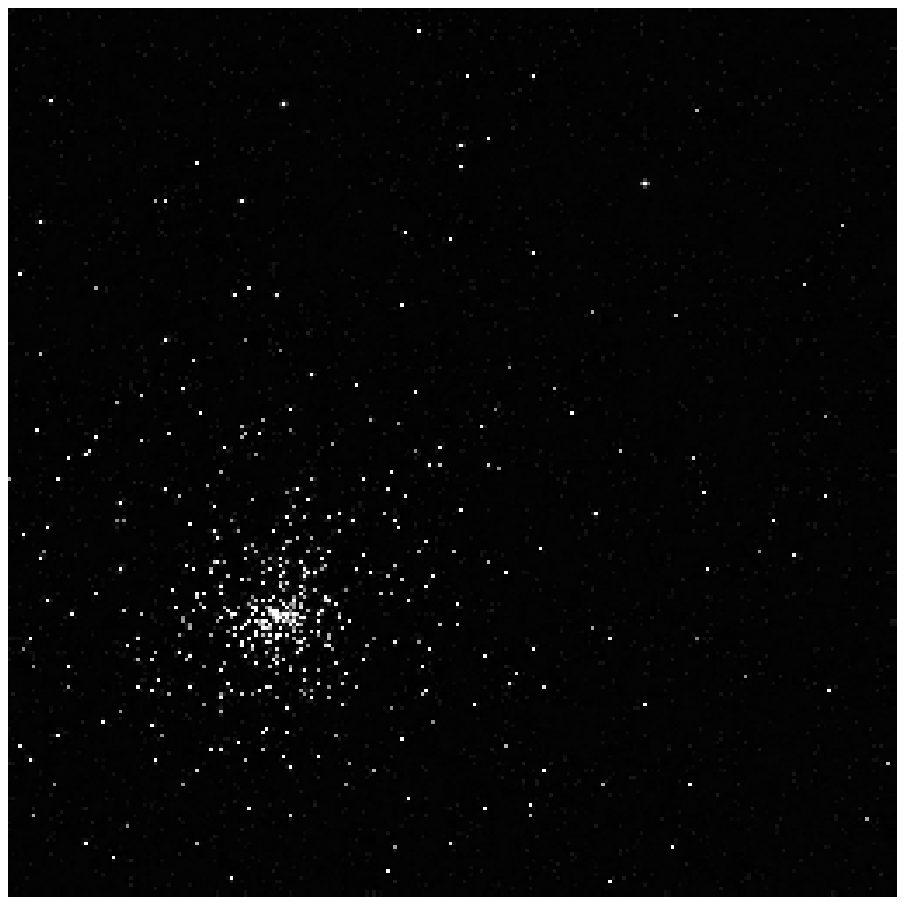} & 
\includegraphics[width=6.2cm]{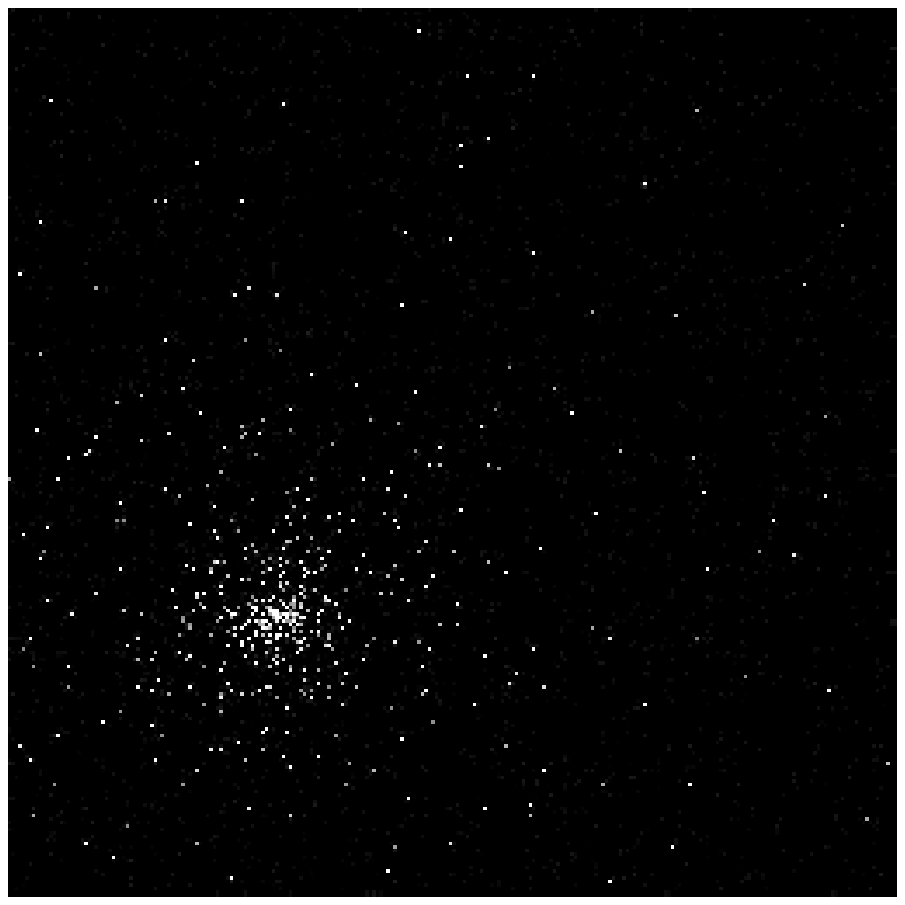}
\end{tabular}\\
\begin{tabular}{cccc}
\hspace{-0.6cm}{\small \textbf{(a')}} & \hspace{-0.6cm}{\small \textbf{(c')}} & \hspace{-0.6cm}{\small \textbf{(d')}} & \hspace{-0.6cm}{\small \textbf{(f')}}\\
\hspace{-0.6cm}\includegraphics[width=4.0cm]{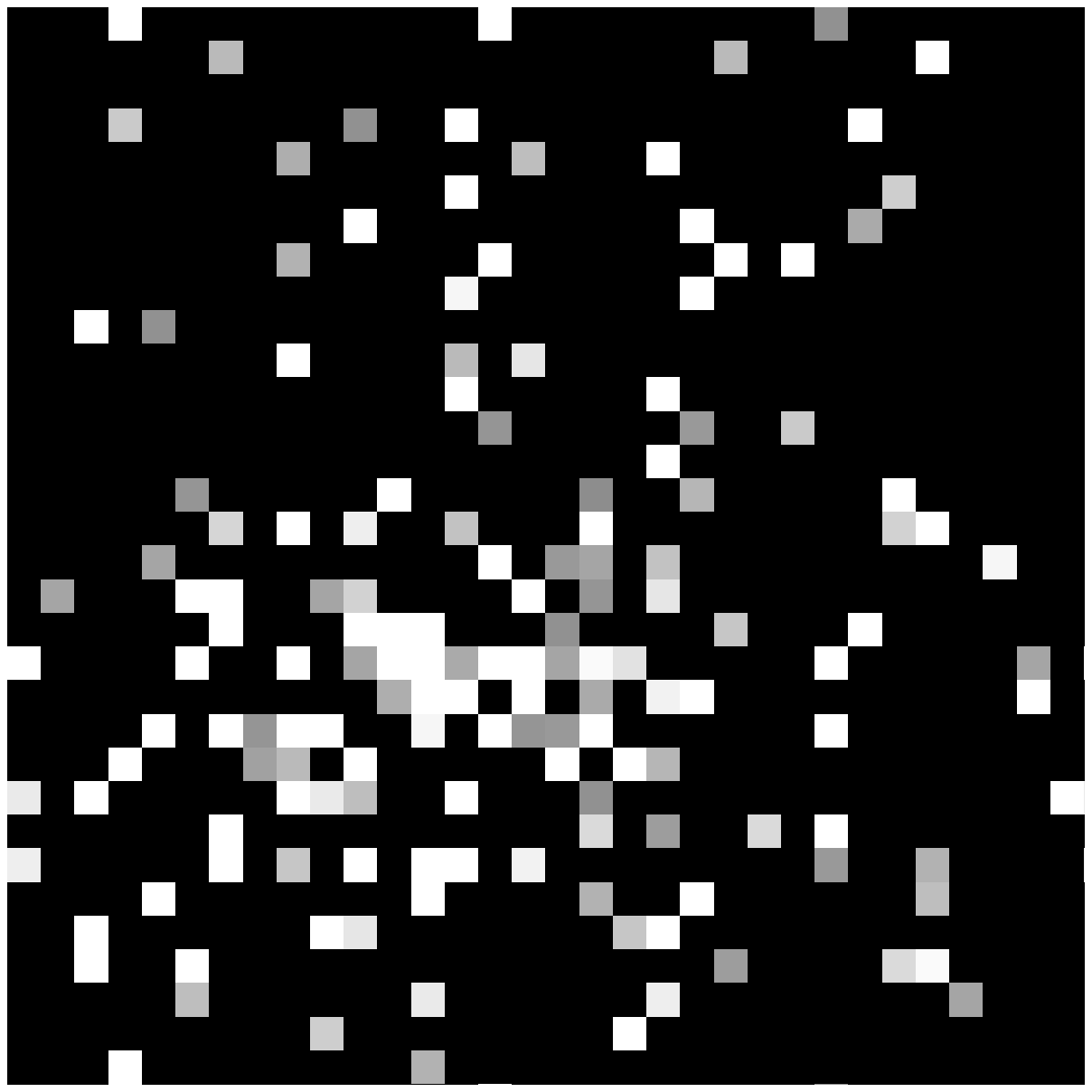} & 
\hspace{-0.6cm}\includegraphics[width=4.0cm]{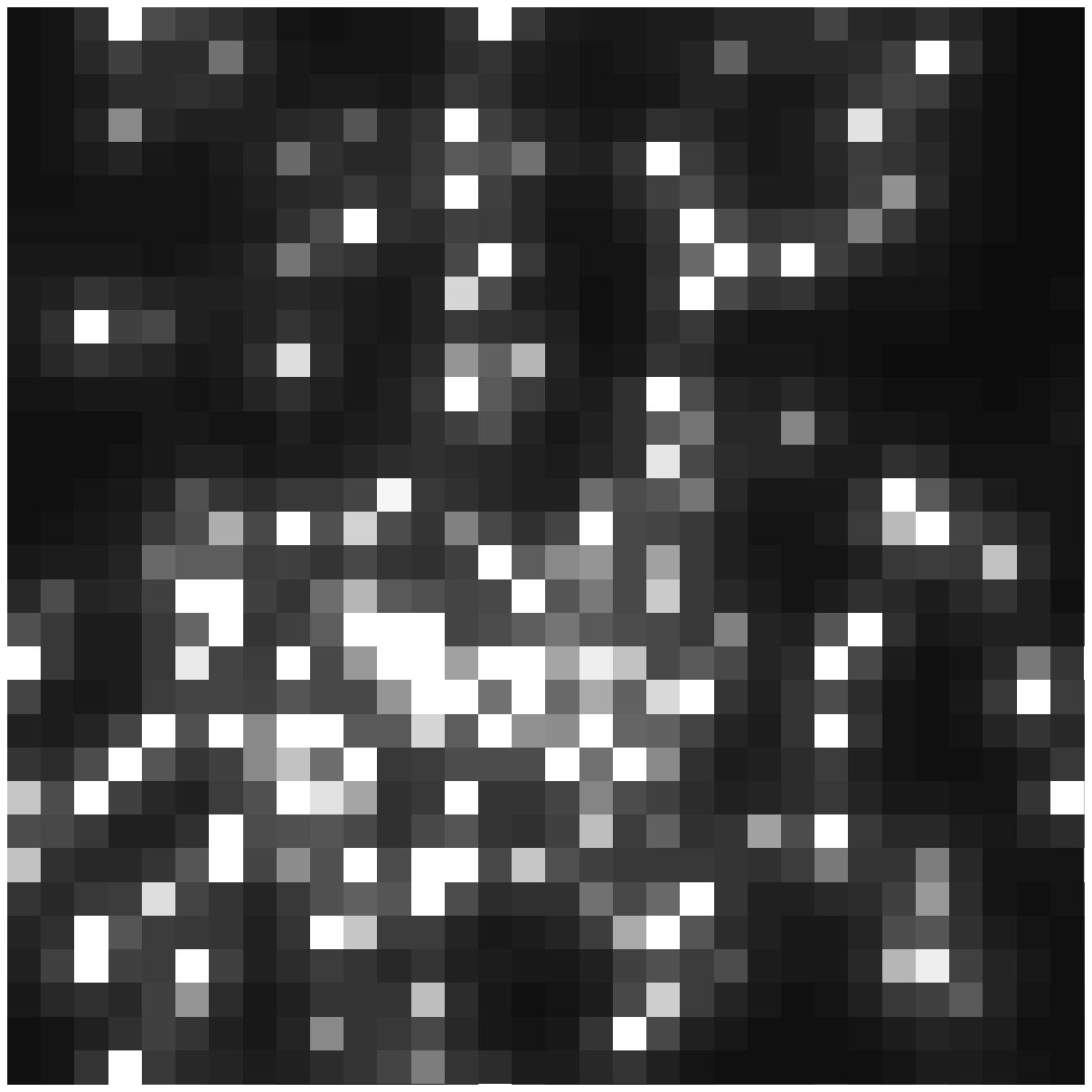} & 
\hspace{-0.6cm}\includegraphics[width=4.0cm]{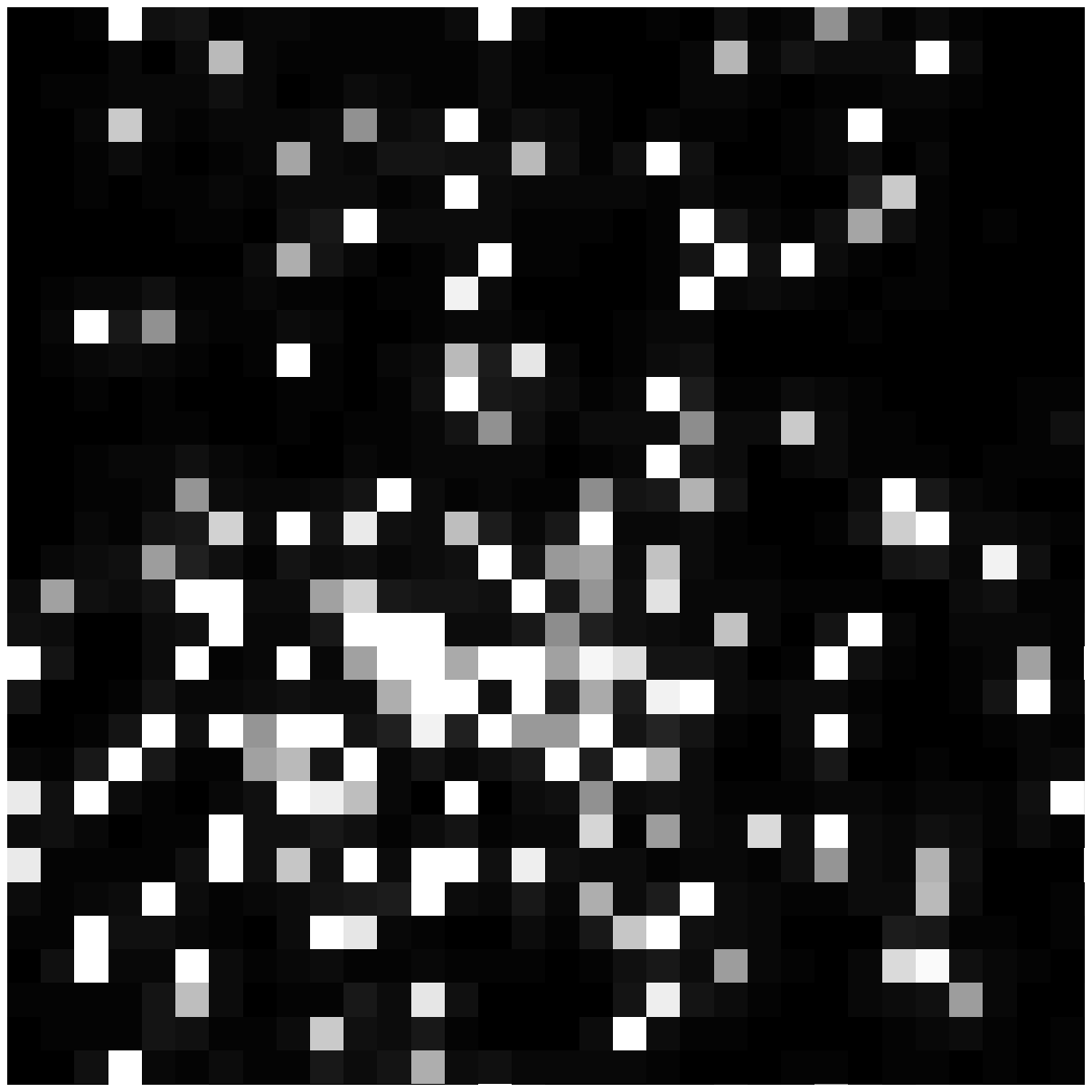} & 
\hspace{-0.6cm}\includegraphics[width=4.0cm]{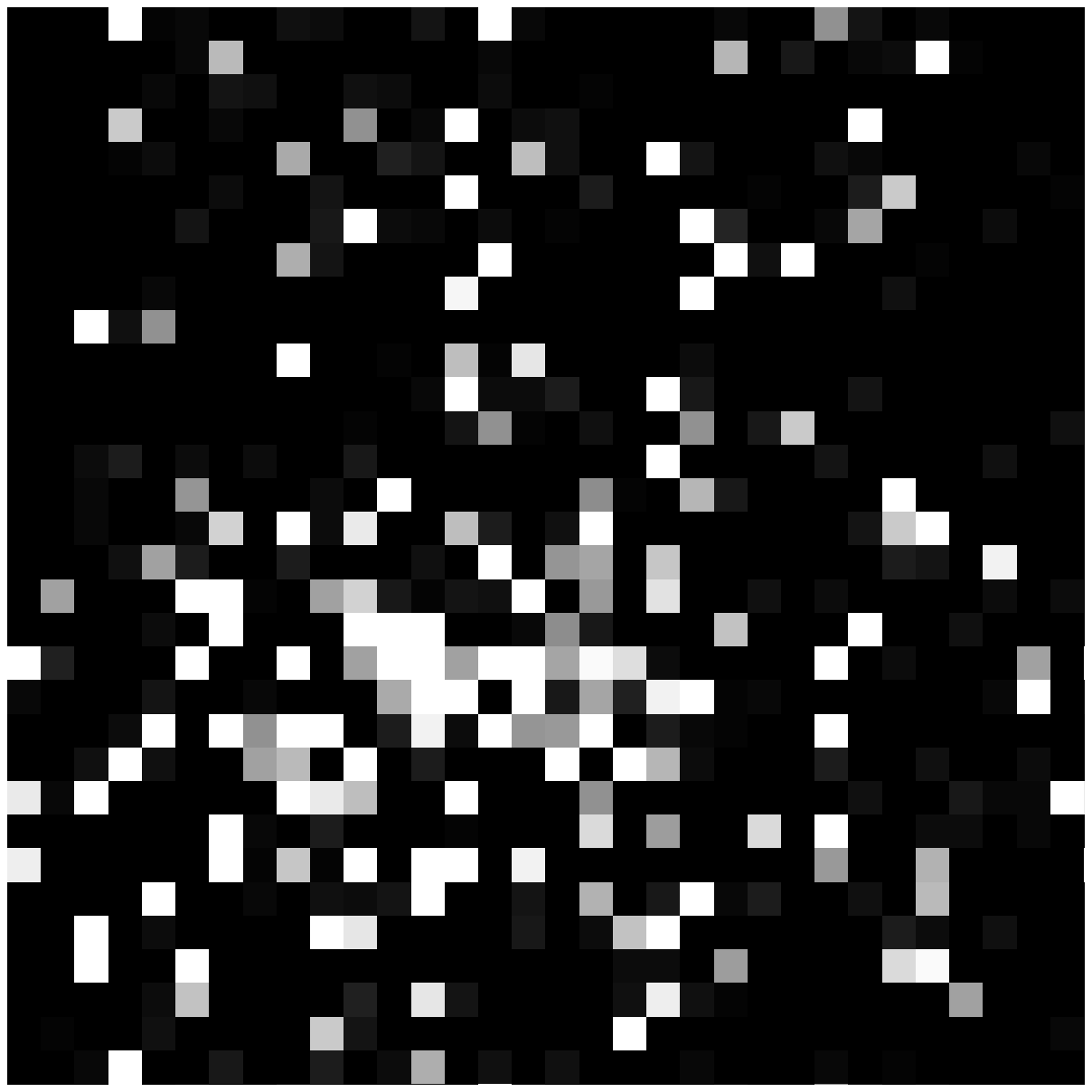}
\end{tabular}
\caption{Images related to the \emph{\texttt{star\_cluster}} test problem, with $\weps=10^{-2}$. Relative errors are reported within parentheses.\textbf{(a)} Exact image. \textbf{(b)} Blurred and noisy image. Restorations obtained at the $200$th iteration of the: \textbf{(c)} MRNSD method $(0.0489)$, \textbf{(d)} NN-FCGLS method $(0.0029)$, \textbf{(e)} ReSt NNCG method $(0.0067)$, and \textbf{(f)} MFISTA method $(0.0070)$. \textbf{(a')}, \textbf{(c')}, \textbf{(d')}, \textbf{(f')}: blow-up of a portion of the images displayed in \textbf{(a)}, \textbf{(c)}, \textbf{(d)}, \textbf{(f)}, respectively.}
\label{fig:Img_StClust}
\end{figure}
\cref{fig:RelErr_StClust} displays the error history for a single run of various methods. The NN-FCGLS iteration satisfying the first criterion in (\ref{stopC}) is marked by an asterisk, while the NN-FCGLS iteration satisfying the second criterion in (\ref{stopC}) is marked by an hexagram. 
The stopping criterion provided in \cite{CLRS} for the ReSt NNCG method is based on the discrepancy principle, and it would prescribe to stop at the 78th total iteration: one can clearly see that immediate restarts happen after this iteration. The qualitative reason behind this phenomenon is that the CGLS residual (i.e., the discrepancy) stabilizes after a few outer iterations, even when incorporating the nonnegative initial guess: at this point, after just one inner iteration, the method is restarted and the decrease in the relative error is slower (with a rate comparable to the NNSD method). The opposite phenomenon occurs for the restarts of the NN-FCGLS method: indeed, restarts are more frequent during the early iterations, i.e., when {there are more changes in the} preconditioning matrix $X^{(m)}$ in (\ref{LeftPrecLS}). One can also clearly see that the behavior of the relative errors of MFISTA is not monotonic at all, as only the residuals $\|r_m\|_2$ are required to decrease monotonically. Moreover, a slight \tql semi-convergence\tqr\ can be detected in the later iterations of NN-FCGLS: a similar phenomenon would appear also for the other solvers, provided that some more iterations are performed. Looking at \cref{tab:Img_StClust}, it should be also remarked that applying MFISTA($1/t$) with an initial large stepsize $t$ results in a faster convergence, i.e., a lower relative error can be attained with the same number of iterations. However, most likely, additional (inner) steps should be performed for backtracking, so that the time per iteration and the overall computational time might be higher than MFISTA($1/t$) with a small $t$. Moreover, though the relative errors delivered by NN-FCGLS and MFISTA are comparable, the latter requires a larger number of iterations, and each iteration is on average slower than any NN-FCGLS iteration (this is clear looking at \cref{fig:RelErr_StClust}). 
\cref{fig:Img_StClust} shows the exact and acquired images, and some restorations obtained at the 200th iteration of various methods. The bottom images in \cref{fig:Img_StClust} are blow-ups of some of the upper images, and it is evident that the image computed by the NN-FCGLS method has less ringing artifacts around the white dots (stars), though the restoration computed by MFISTA has similar quality. In this experiment, as well as some of the following ones, the reconstructed images are displayed after performing a fixed number of iterations, so that the different progress of each method can be visually compared.

\textbf{Example 2.} The second set of experiments deals with the deblurring and denoising of the well-known \texttt{satellite} test image of size $256\times 256$ pixels \cite{RestTools}. Two different matrices $A$ are taken into account, which model spatially invariant atmospheric blur. In the first case, only Gaussian noise of level $\weps=10^{-1}$ is added, while in the second case both Gaussian and Poisson noise of total level around $\weps=1.5\cdot 10^{-2}$ is added. 
In both cases, the parameter $\mmaxin$ of Algorithm \ref{alg:NN-FCGLS} is set to 20, while the parameter $\tau$ appearing in (\ref{stopC}) is set to $10^{-4}$. The first tests consider only Gaussian noise. The left frame of \cref{fig:NN-FCGLS_sat} displays the effect of different initial guesses $x_0^0$ on the relative errors computed by the NN-FCGLS method, with full recurrences for the updates of $d_m^{k-1}$. The vectors $x_0^0$ used for this experiment are the projection of the data vector $b$ onto the nonnegative orthant (denoted by $b_0$), and the projection of the vector $A^Tb$ onto the nonnegative orthant (denoted by $(A^Tb)_0$); also the option $x_0^0=\bzeros$ is considered and, to avoid immediate stagnation of NN-FCGLS, the identity matrix of order $N$ is set as first \tql preconditioner\tqr\ $X^{(0)}$ in (\ref{LeftPrecLS}). The right frame of \cref{fig:NN-FCGLS_sat} considers the effect of different values of the truncation parameter $\hat{m}$ on the relative errors computed by the NN-FCGLS method, with $x_0^0=b_0$. When updating the descent direction $d_m^{k-1}$ in Algorithm \ref{alg:NN-FCGLS} (see also steps \ref{step:ind1} and \ref{step:ind2} of Algorithm \ref{alg:FCGLS}), the following values are used: $\hat{m}=\mmaxin$ (corresponding to a full recurrence), $\hat{m}=1$ (corresponding to a CGLS-like recurrence), and $\hat{m}=5$. \cref{tab:NN-FCGLS_sat} summarizes the results obtained running these first experiments: more precisely, the values of the minimum relative error (\ref{RelErr}), and the corresponding $m$ are reported.
\begin{figure}[htbp] 
\centering
\begin{tabular}{cc}
\hspace{-0.4cm}{\small \textbf{(a)}} & \hspace{-0.7cm}{\small \textbf{(b)}}\\ 
\hspace{-0.4cm}\includegraphics[width=7.8cm]{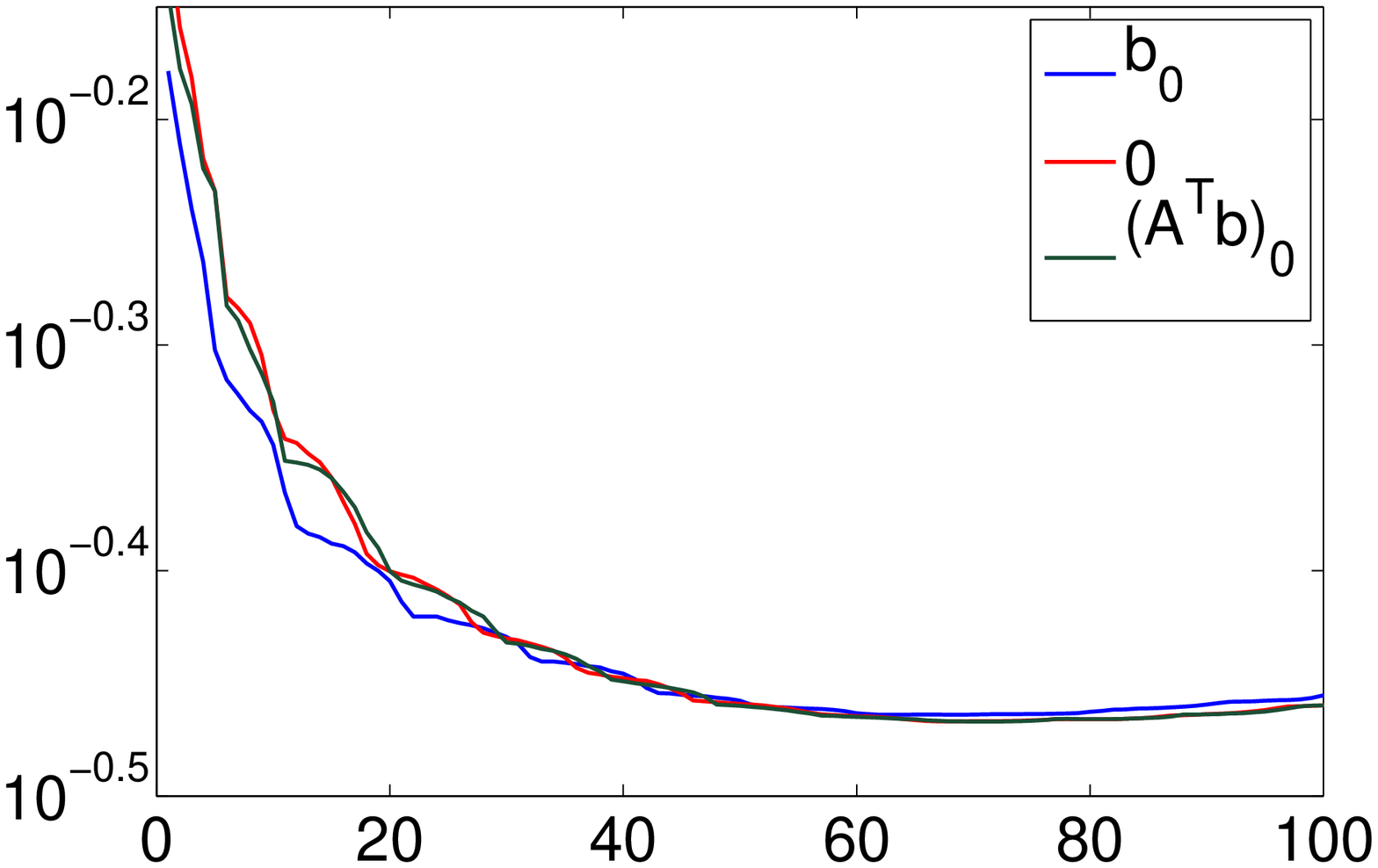} & 
\hspace{-0.7cm}\includegraphics[width=7.8cm]{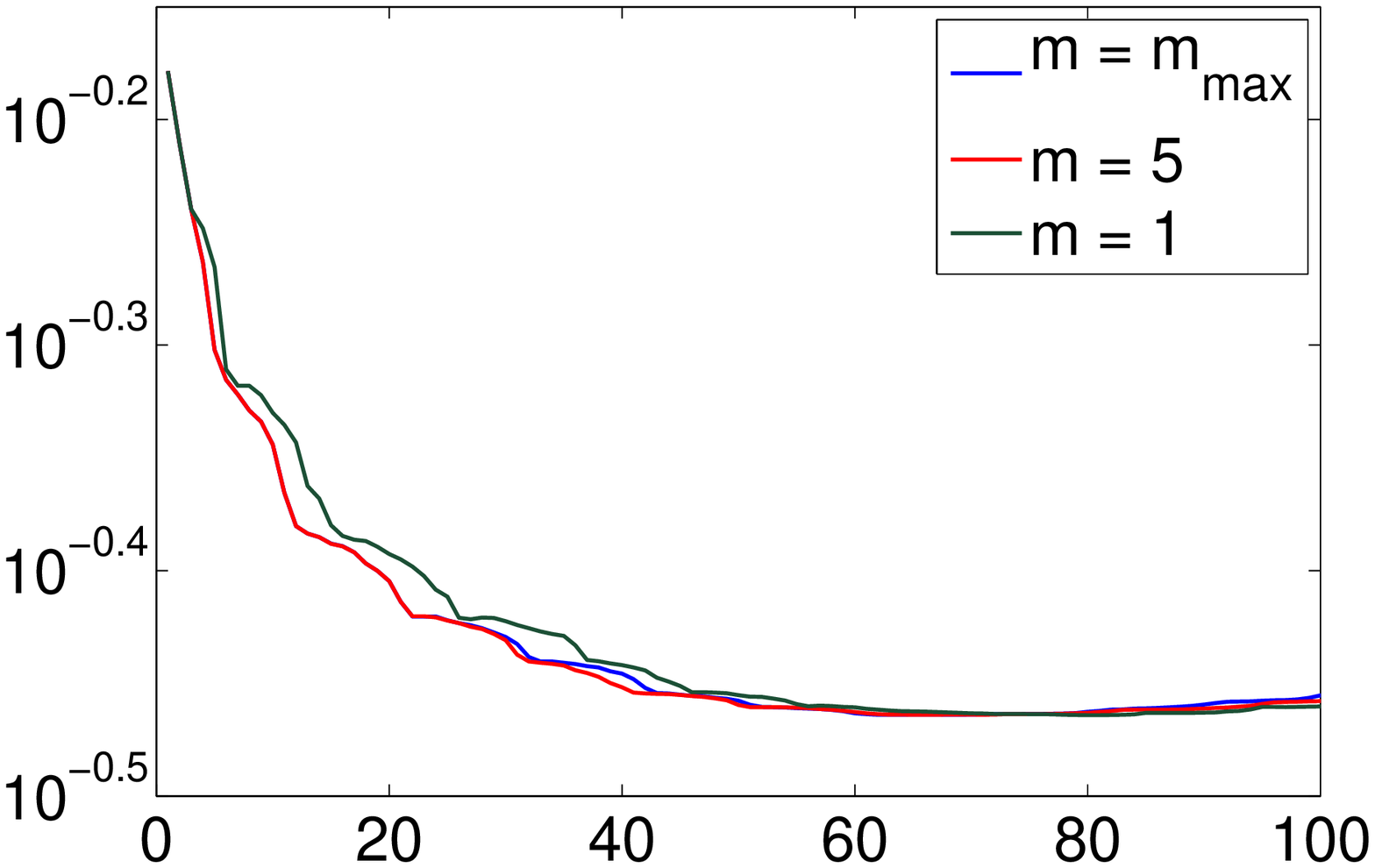}
\end{tabular}
\caption{\emph{\texttt{satellite}} test problem, with Gaussian noise of level $\widetilde{\eps}=10^{-1}$. History of the relative errors of the NN-FCGLS method, varying the initial guess $x_0^0$ (frame \textbf{(a)}), and varying the truncation parameter $\hat{m}$ for the update of $d_m^{k-1}$ (frame \textbf{(b)}).}
\label{fig:NN-FCGLS_sat}
\end{figure}
\begin{table}
\footnotesize
\centering
\caption{\emph{\texttt{satellite}} test problem, with Gaussian noise of level $\widetilde{\eps}=10^{-1}$. Minimum relative error achieved by running NN-FCGLS with different combinations of $x_0^0$ and $\hat{m}$.}
\label{tab:NN-FCGLS_sat}
\begin{tabular}{|cccc|}
\hline
 {\textbf{rel.error}} & {\textbf{iterations}} & {\textbf{$x_0^0$}} & {\textbf{$\hat{m}$}}\\\hline
3.4363e-01 & 63 & $b_0$ & $\mmaxin$\\ \hline
3.4131e-01 & 68 & $\bzeros$ & $\mmaxin$\\ \hline
3.4121e-01 & 72 & $(A^Tb)_0$ & $\mmaxin$\\ \hline
3.4363e-01 & 65 & $b_0$ & 5\\ \hline
3.4347e-01 & 80 & $b_0$ & 1\\ \hline
\end{tabular}
\end{table}
Looking at \cref{tab:NN-FCGLS_sat}, one can conclude that NN-FCGLS is very robust with respect to both the choices of $x_0^0$ and $\hat{m}$. The effect of different choices of $x_0^0$ is mostly evident during the early iterations, when $x_0^0=b_0$ seems to outperform the other options (this is quite common when considering image deblurring problems). Moreover, though $\hat{m}=\mmaxin$ should be chosen according to the theory of FCGLS, also lower values of $\hat{m}$ can deliver results of the same quality, except for the early iterations. In the following tests performed with the \texttt{satellite} test image, the choices $x_0^0=b_0$ and $\hat{m}=\mmaxin$ will be considered. The vector $b_0$ is taken as initial guess for all the solvers. 

The next tests compare NN-FCGLS with the ReSt NNCG, (M)FISTA($1/t$), MRNSD, and NNSD methods. Following the suggestions in \cite{BerNagy13, NS00}, also a preconditioned version of MRNSD (dubbed \tql PMRNSD\tqr) is taken into account. A special (fixed) preconditioner for PMRNSD is computed at the beginning of the iterations by exploiting an approximation of the singular value decomposition of $A$ by means of the fast Fourier transform. The smaller approximate singular values are set to one, so that they are not inverted when applying preconditioning.
\cref{tab:Sat} summarizes the average results over 10 runs of the \texttt{satellite} test problem, with different realizations of the of the random noise. All the methods are stopped after 200 (total) iterations. 
\cref{fig:RelErr_Sat} displays the history of the relative errors for some of the methods considered in \cref{tab:Sat}; for this test, the NN-FCGLS iterations satisfying the first and the second stopping criteria in (\ref{stopC}) coincide, and they are highlighted by big markers. The \tql semi-convergence\tqr\ phenomenon is evident for both the PMRNSD and NN-FCGLS methods, which are the fastest solvers for this test problem. 
\begin{figure}[htbp]
\centering
\includegraphics[width=14.0cm]{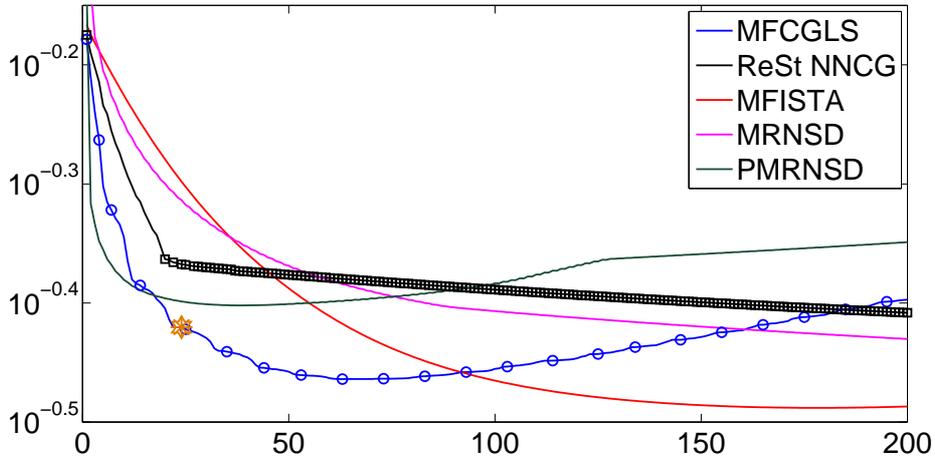}
\caption{\emph{\texttt{satellite}} test problem, with Gaussian noise of level $\widetilde{\eps}=10^{-1}$. History of the relative errors. For the NN-FCGLS and ReSt CG methods, a small marker is used to emphasize the iterations where restarts happen. The big markers for NN-FCGLS highlight the stopping iterations (which coincide, in this case).}
\label{fig:RelErr_Sat}
\end{figure}
\begin{table}
\footnotesize
\centering
\caption{\emph{\texttt{satellite}} test problem, with Gaussian noise of level $\widetilde{\eps}=10^{-1}$. Average values over ${10}$ runs of the test problem, with different noise realizations.}
\label{tab:Sat}
\begin{tabular}{|lcccc|}
\hline
 & {\textbf{rel.error}} & {\textbf{iterations}} & {\textbf{tot.time}} & {\textbf{av.time}}\\\hline
{\textbf{NN-FCGLS}} & 3.5098e-01 & 70.33 & 5.49 & 0.08\\ \hline
{\textbf{ReSt NNCG}} & 4.0957e-01 & 106.67 & 9.38 & 0.08\\ \hline
{\textbf{FISTA}} & 3.2969e-01 & 164.33 & 21.22 & 0.12\\ \hline
{\textbf{MFISTA}} & 3.2583e-01 & 177.00 & 23.10 & 0.13\\ \hline
{\textbf{MFISTA(0.2)}} & 3.3318e-01 & 137.00 & 20.58 & 0.15\\ \hline
{\textbf{MFISTA(5)}} & 3.3397e-01 & 200.00 & 26.86 & 0.13\\ \hline
{\textbf{MRNSD}} & 3.7720e-01 & 200.00 & 12.55 & 0.06\\ \hline
{\textbf{PMRNSD}} & 4.0032e-01 & 37.33 & 2.62 & 0.07\\ \hline
{\textbf{NNSD}} & 4.3095e-01 & 200.00 & 13.82 & 0.07\\ \hline
\end{tabular}
\end{table}
\begin{figure}[htbp]
\centering
\begin{tabular}{ccc}
\hspace{-0.9cm}{\small \textbf{(a)}} & \hspace{-0.9cm}{\small \textbf{(b)}} & \hspace{-0.9cm}{\small \textbf{(c)}} \\ 
\hspace{-0.9cm}\includegraphics[width=5.6cm]{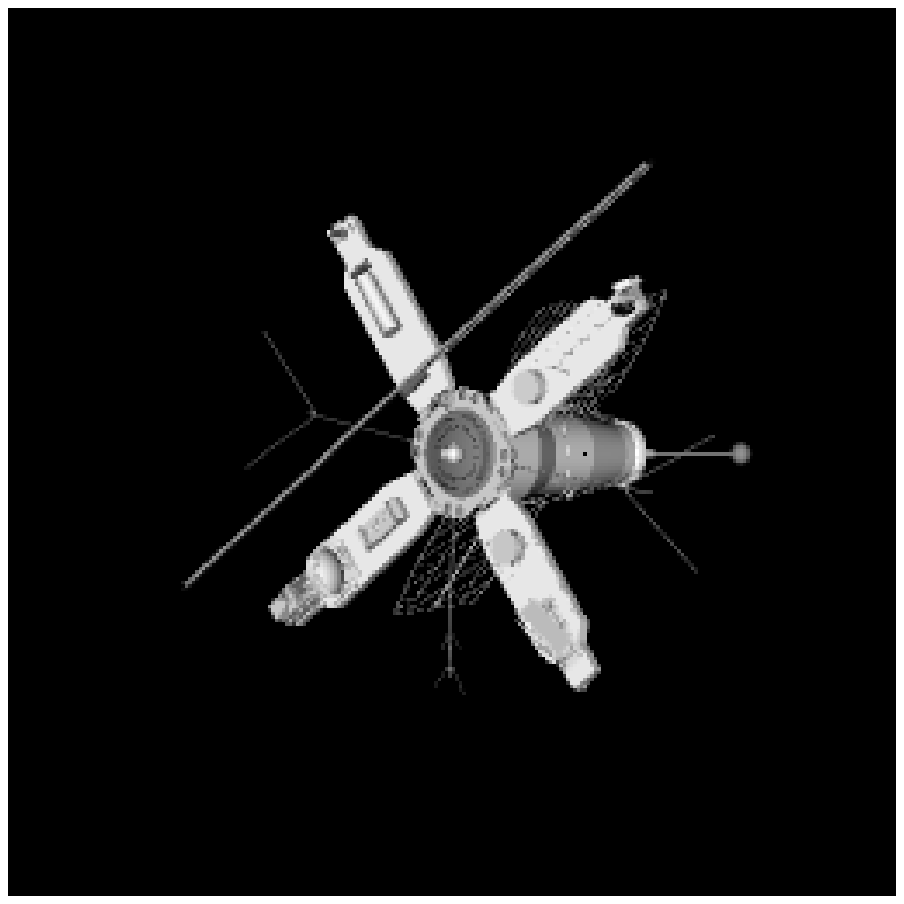} & 
\hspace{-0.9cm}\includegraphics[width=5.6cm]{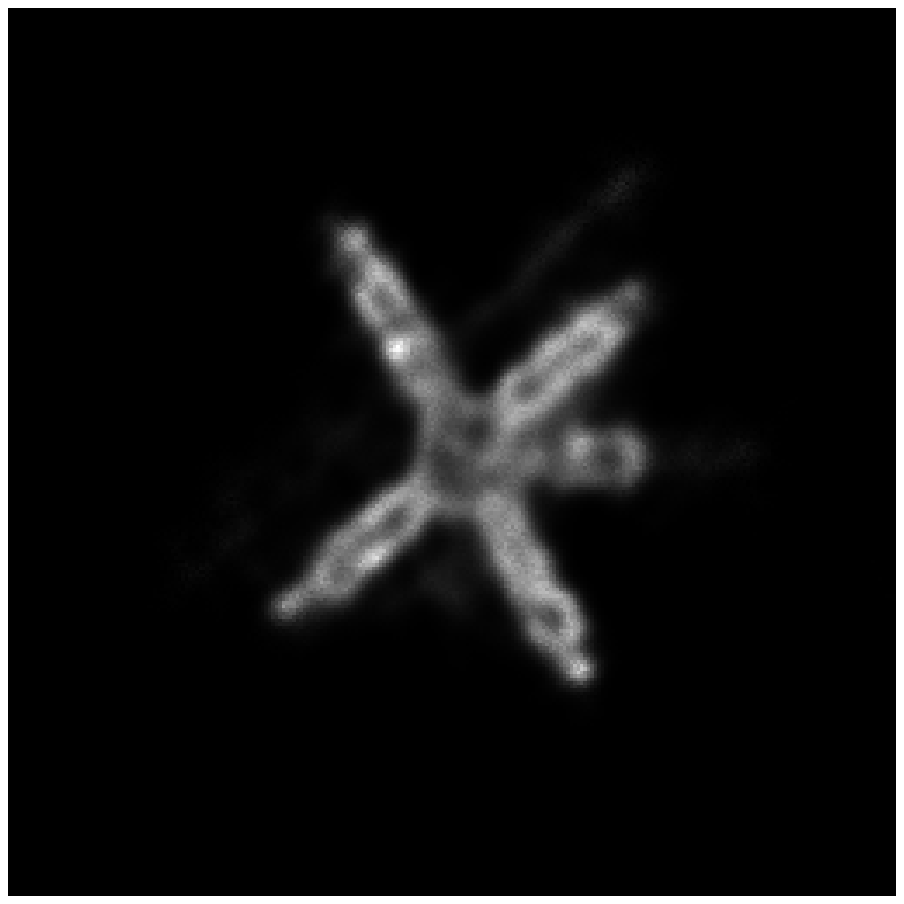} & 
\hspace{-0.9cm}\includegraphics[width=5.6cm]{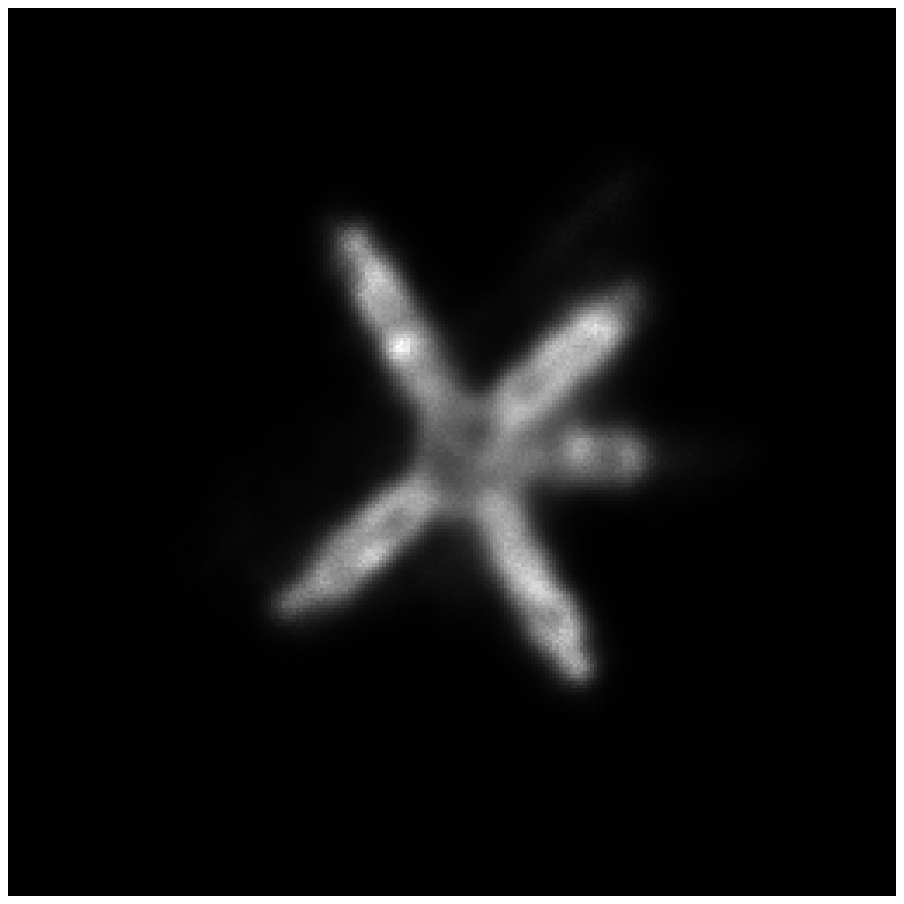} \\
\hspace{-0.9cm}{\small \textbf{(d)}} & \hspace{-0.9cm}{\small \textbf{(e)}} & \hspace{-0.9cm}{\small \textbf{(f)}} \\ 
\hspace{-0.9cm}\includegraphics[width=5.6cm]{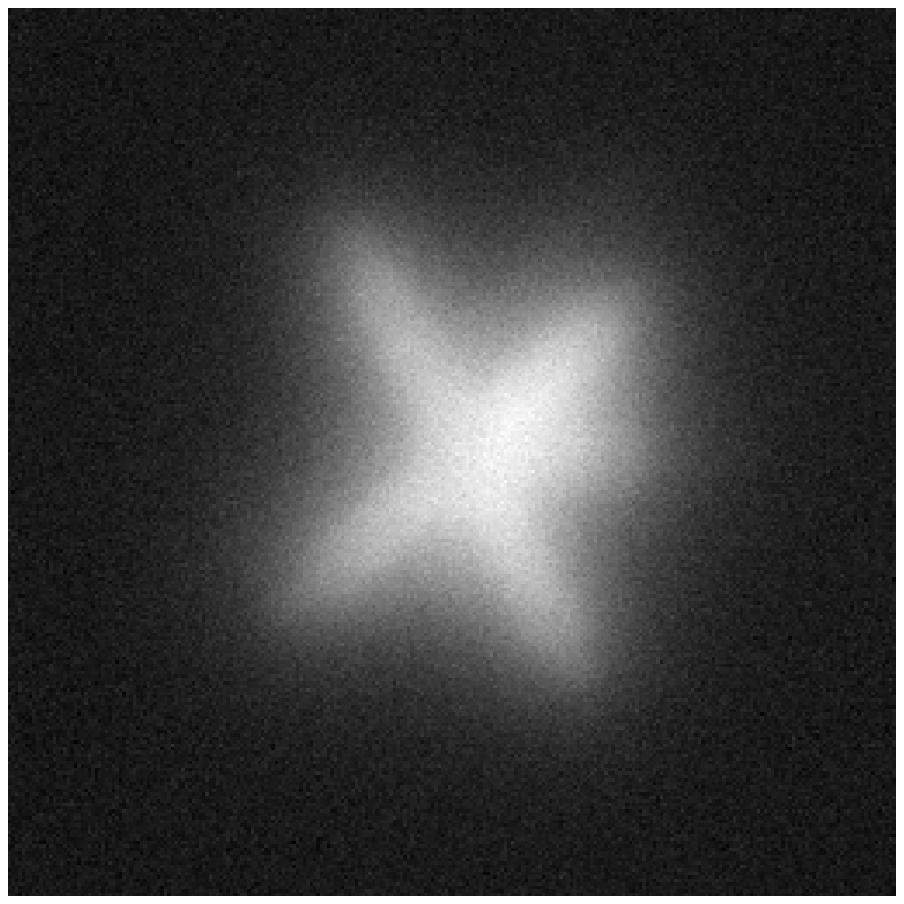} & 
\hspace{-0.9cm}\includegraphics[width=5.6cm]{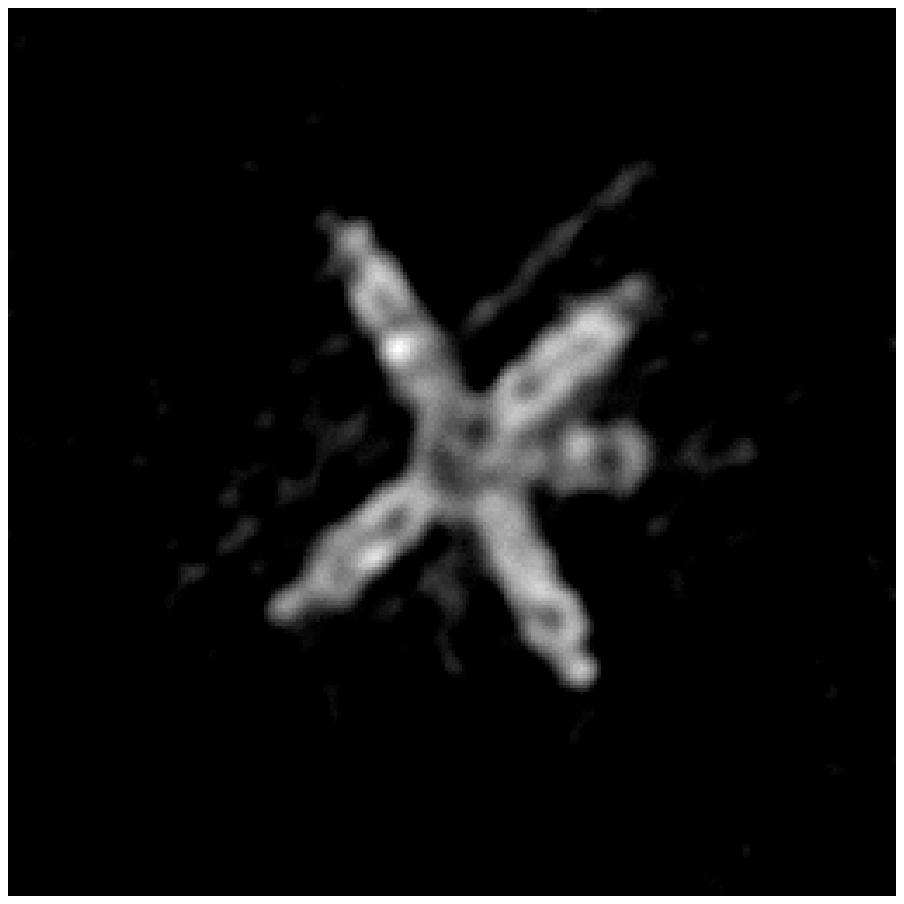} & 
\hspace{-0.9cm}\includegraphics[width=5.6cm]{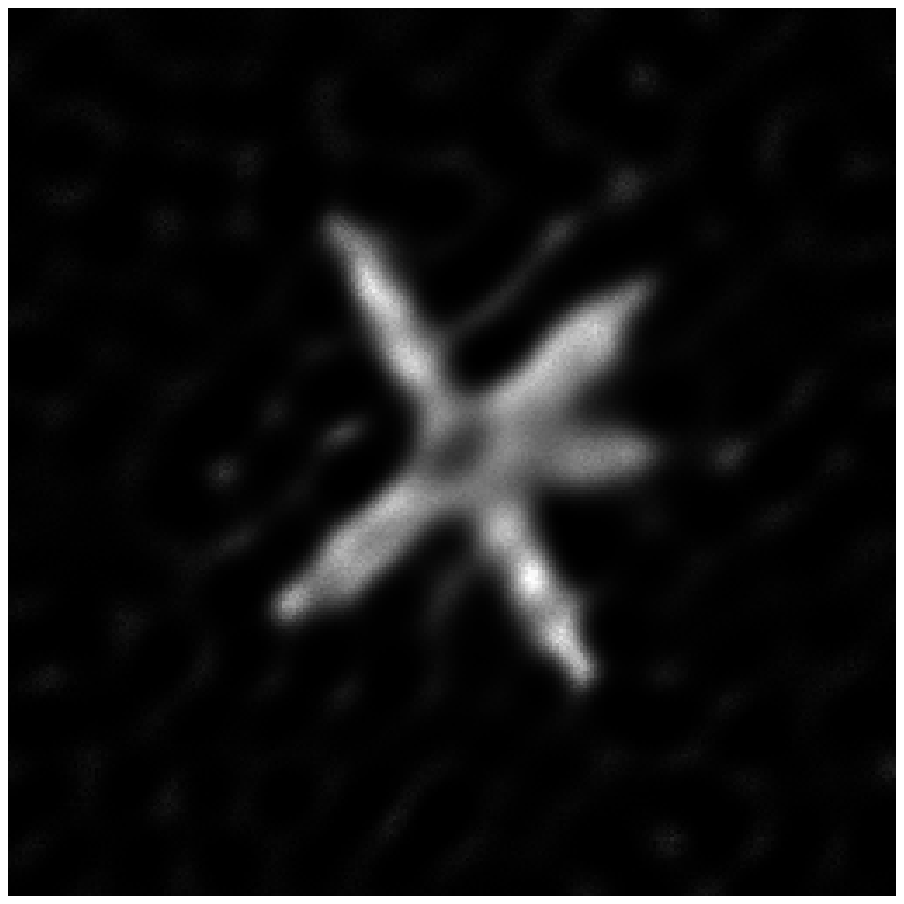}
\end{tabular}%
\caption{Images related to the \emph{\texttt{satellite}} test problem, with Gaussian noise of level $\widetilde{\eps}=10^{-1}$. Relative errors are reported within parentheses. \textbf{(a)} Exact image. \textbf{(b)} Restoration by NN-FCGLS, iteration \# $63$ $(0.3436)$. \textbf{(c)} Restoration by MRNSD, iteration \# $200$ $(0.3711)$. \textbf{(d)} Blurred and noisy image. \textbf{(e)} Restoration by MFISTA, iteration \# $200$ $(0.3259)$. \textbf{(f)} Restoration by PMRNSD, iteration \# $38$ $(0.3962)$.}
\label{fig:Img_Sat}
\end{figure}
\begin{table}
\footnotesize
\centering
\caption{\emph{\texttt{satellite}} test problem, with both Gaussian and Poisson noise of level around $\weps=1.5\cdot 10^{-2}$. Average values over $10$ runs of the test problem, with different noise realizations.}
\label{tab:SatPoiss}
\begin{tabular}{|lcccc|}
\hline
 & {\textbf{rel.error}} & {\textbf{iterations}} & {\textbf{tot.time}} & {\textbf{av.time}}\\\hline
{\textbf{CP-NN-FCGLS}} & 1.2785e-01 & 300.00 & 31.65 & 0.08\\ \hline
{\textbf{CP-NN-FCGLS($k$)}} & 1.2778e-01 & 300.00 & 32.17 & 0.08\\ \hline
{\textbf{WMRNSD}} & 1.8201e-01 & 300.00 & 28.34 & 0.09\\ \hline
{\textbf{KWMRNSD}} & 1.3590e-01 & 300.00 & 37.19 & 0.12\\ \hline
\end{tabular}
\end{table}
\begin{figure}[htbp]
\centering
\includegraphics[width=14.0cm]{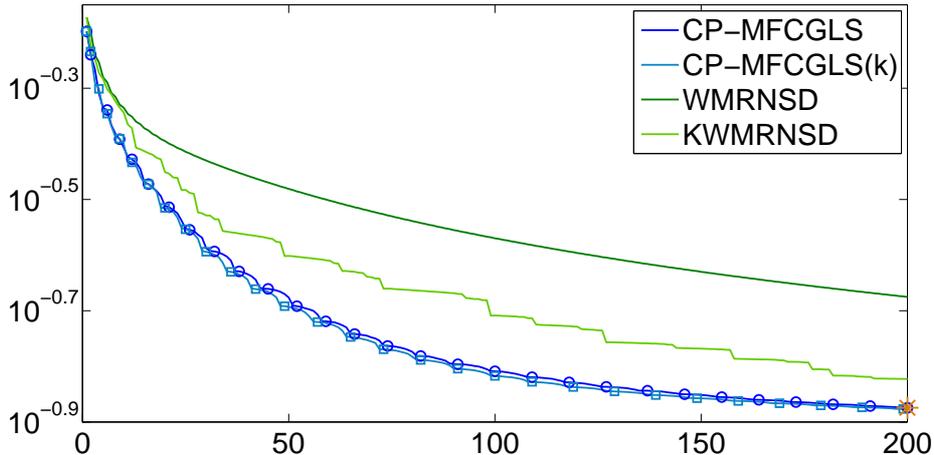}
\caption{\emph{\texttt{satellite}} test problem, with Gaussian and Poisson noise of level around $\weps=1.5\cdot 10^{-2}$. History of the relative errors. For the CP-NN-FCGLS and CP-NN-FCGLS($k$) methods, a small marker is used to emphasize the iterations where restarts happen. The big mark highlights the iteration satisfying the first stopping criterion for CP-NN-FCGLS in (\ref{stopC}): in this case, the maximum number of total iterations is reached.}
\label{fig:RelErr_SatPoiss}
\end{figure}
\begin{figure}[htbp]
\centering
\begin{tabular}{cc}
\hspace{-0.6cm}{\small \textbf{(a)}} & \hspace{-0.6cm}{\small \textbf{(b)}}\\ 
\hspace{-0.6cm}\includegraphics[width=6.2cm]{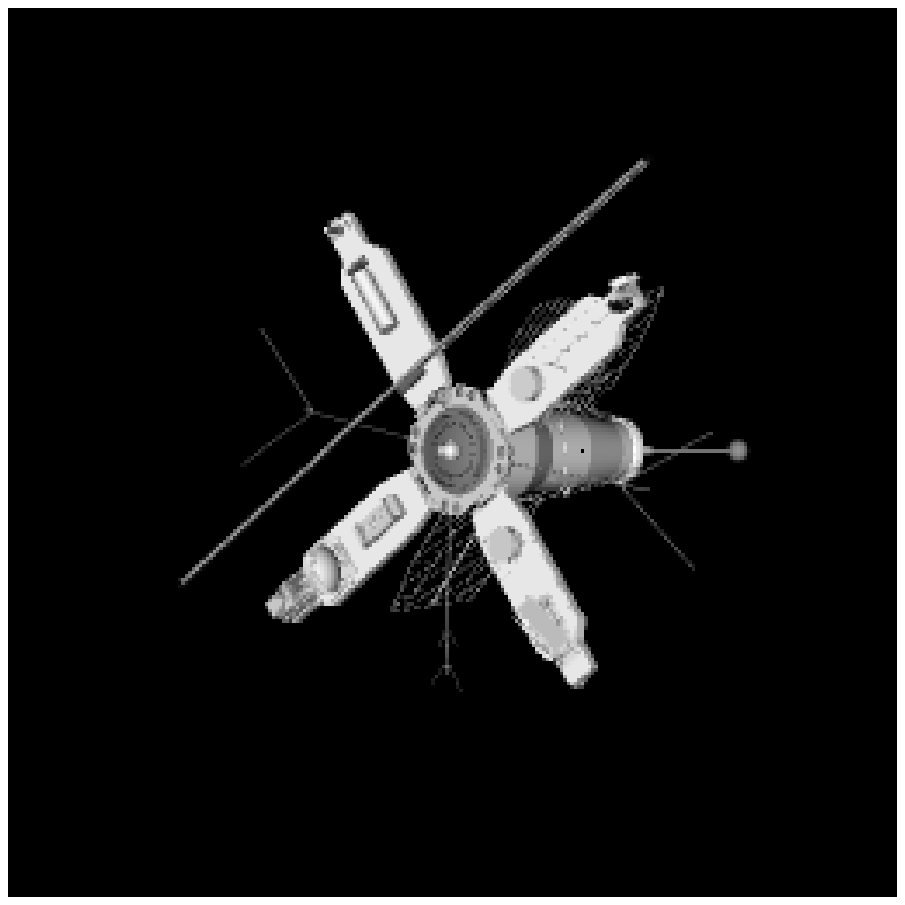} & 
\hspace{-0.6cm}\includegraphics[width=6.2cm]{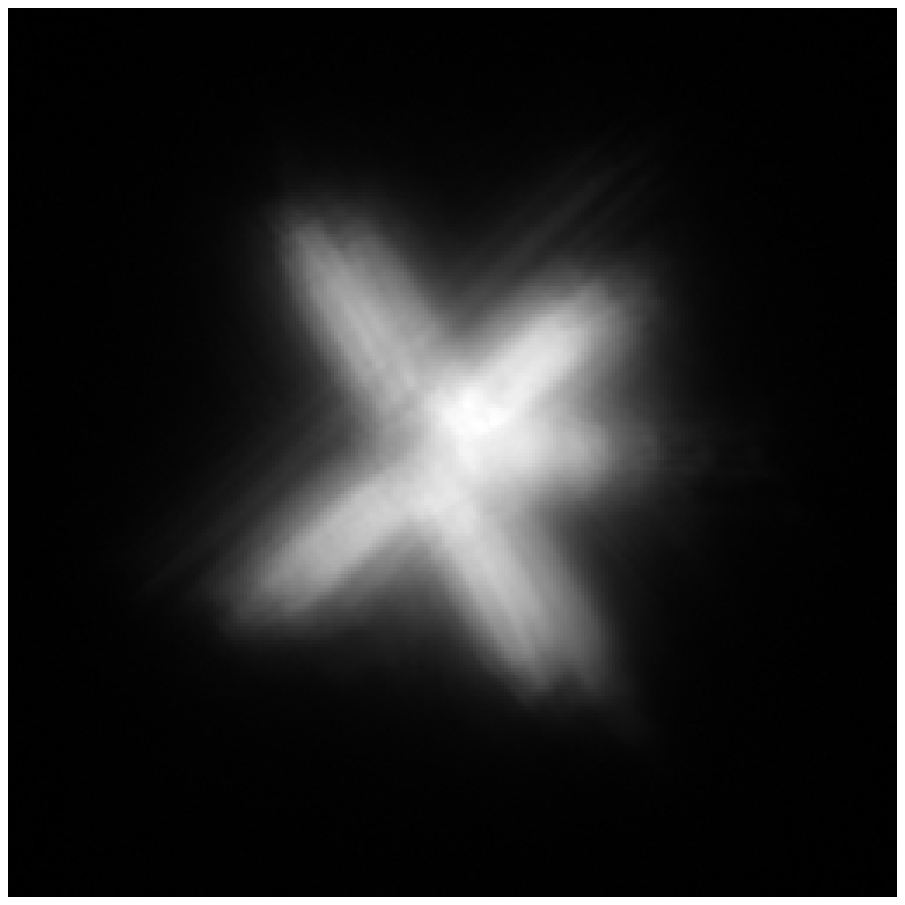}\\
\hspace{-0.6cm}{\small \textbf{(c)}} & \hspace{-0.6cm}{\small \textbf{(d)}}\\ 
\hspace{-0.6cm}\includegraphics[width=6.2cm]{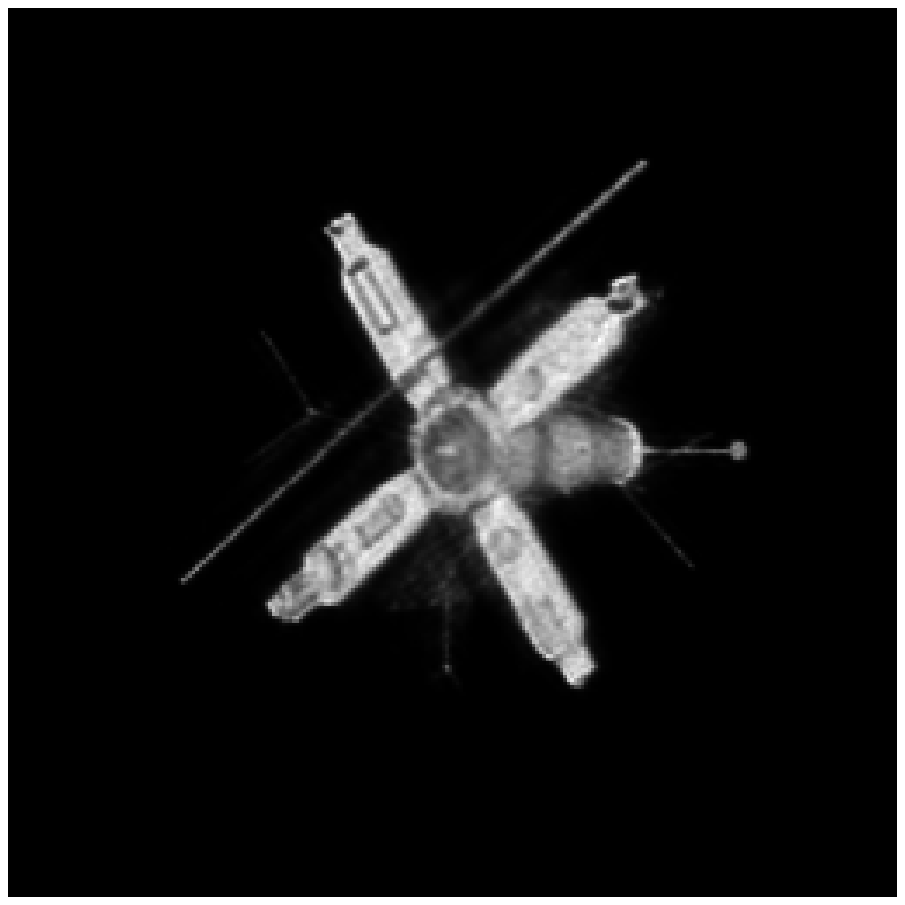} & 
\hspace{-0.6cm}\includegraphics[width=6.2cm]{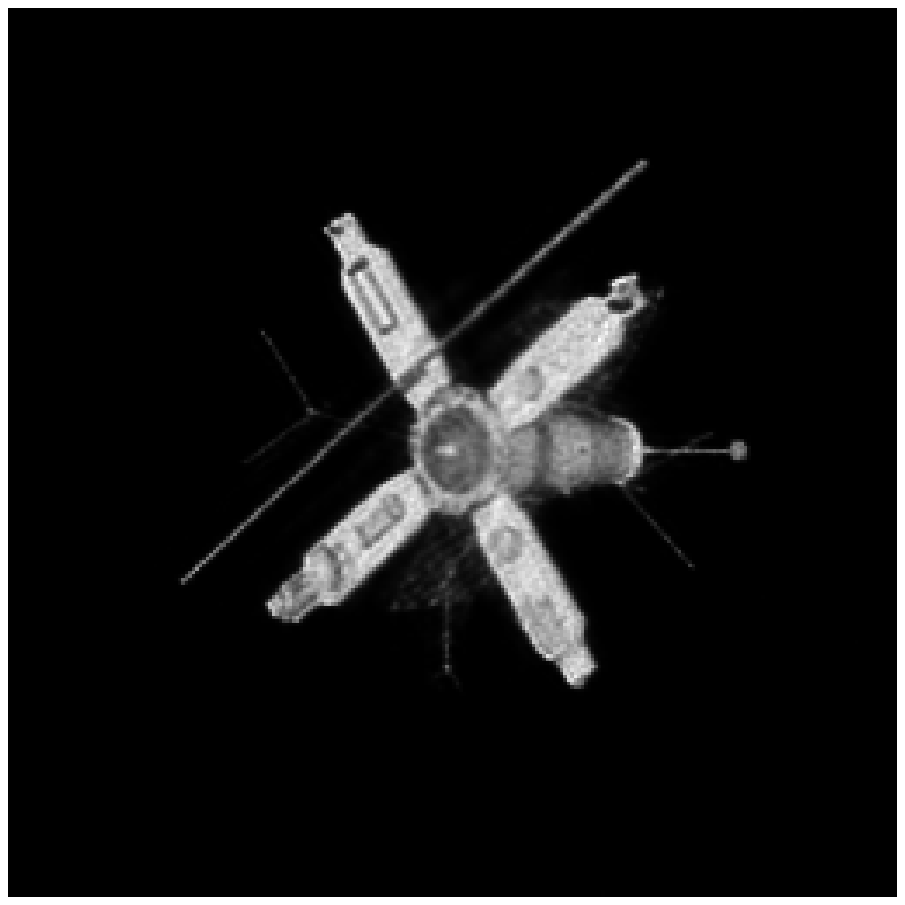}\\
\end{tabular}%
\caption{Images related to the \emph{\texttt{ satellite}} test problem, with both Gaussian and Poisson noise of level around $\weps=1.5\cdot 10^{-2}$. Relative errors are reported within parentheses. \textbf{(a)} Exact image. \textbf{(b)} Blurred and noisy image. Restorations obtained at the ${100}$th iteration of the: \textbf{(c)} KWMRNSD method ${(0.1844)}$, and \textbf{(d)} CP-NN-FCGLS($k$) method ${(0.1563)}$.}
\label{fig:Img_SatPoiss}
\end{figure}
\cref{fig:Img_Sat} displays the restored images of best quality obtained by different methods. One can see MFISTA to deliver slightly better results than NN-FCGLS for this particular example. However, looking at the displayed images, the differences are not huge, and NN-FCGLS is much faster.
The remaining tests are concerned with the deblurring and denoising of the \texttt{satellite} image, perturbed by both Gaussian and Poisson noise. The Gaussian standard deviation parameter is $\sigma = 20$, while the Poisson parameter is $\beta = 60$.  \cref{tab:SatPoiss} compares different methods for solving the nonnegatively constrained covariance-preconditioned problem (\ref{covNNLS}). In particular, the WMRNSD and KWMRNSD methods \cite{BN06} are compared with the new CP-NN-FCGLS method. Regarding the latter, the notation CP-NN-FCGLS($k$) is used to emphasize that the restart-dependent covariance matrix $\covM^{(k)}$ is used in Algorithm \ref{alg:CPNN-FCGLS}.
\cref{fig:RelErr_SatPoiss} displays the history of the relative errors, while \cref{fig:Img_SatPoiss} displays the restorations obtained at the 100th iteration of the KWMRNSD and CP-NN-FCGLS($k$) methods.  
\begin{figure}[htbp]
\centering
\begin{tabular}{c}
\includegraphics[width=14.0cm]{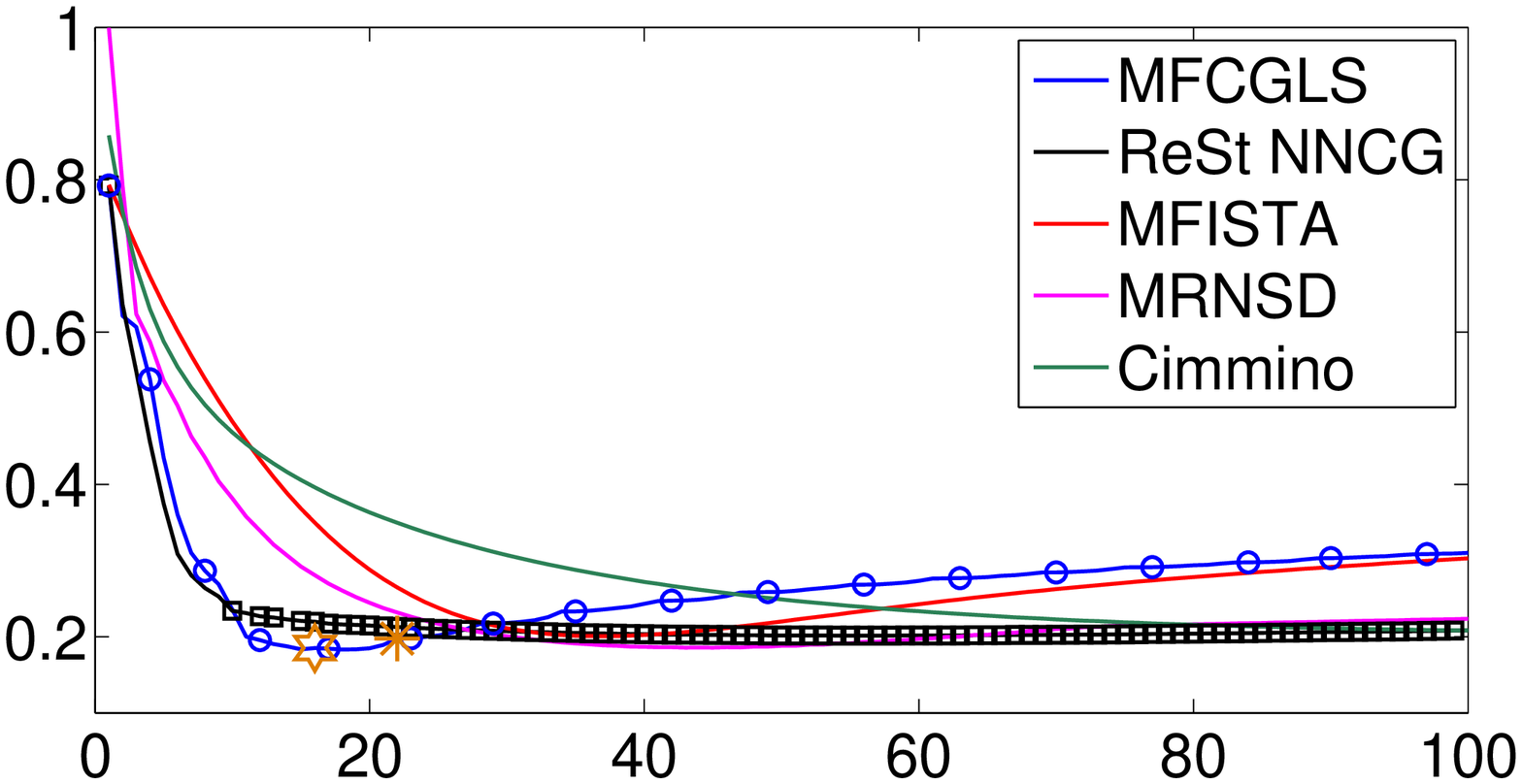}\vspace{-0.2cm}\\
\includegraphics[width=14.0cm]{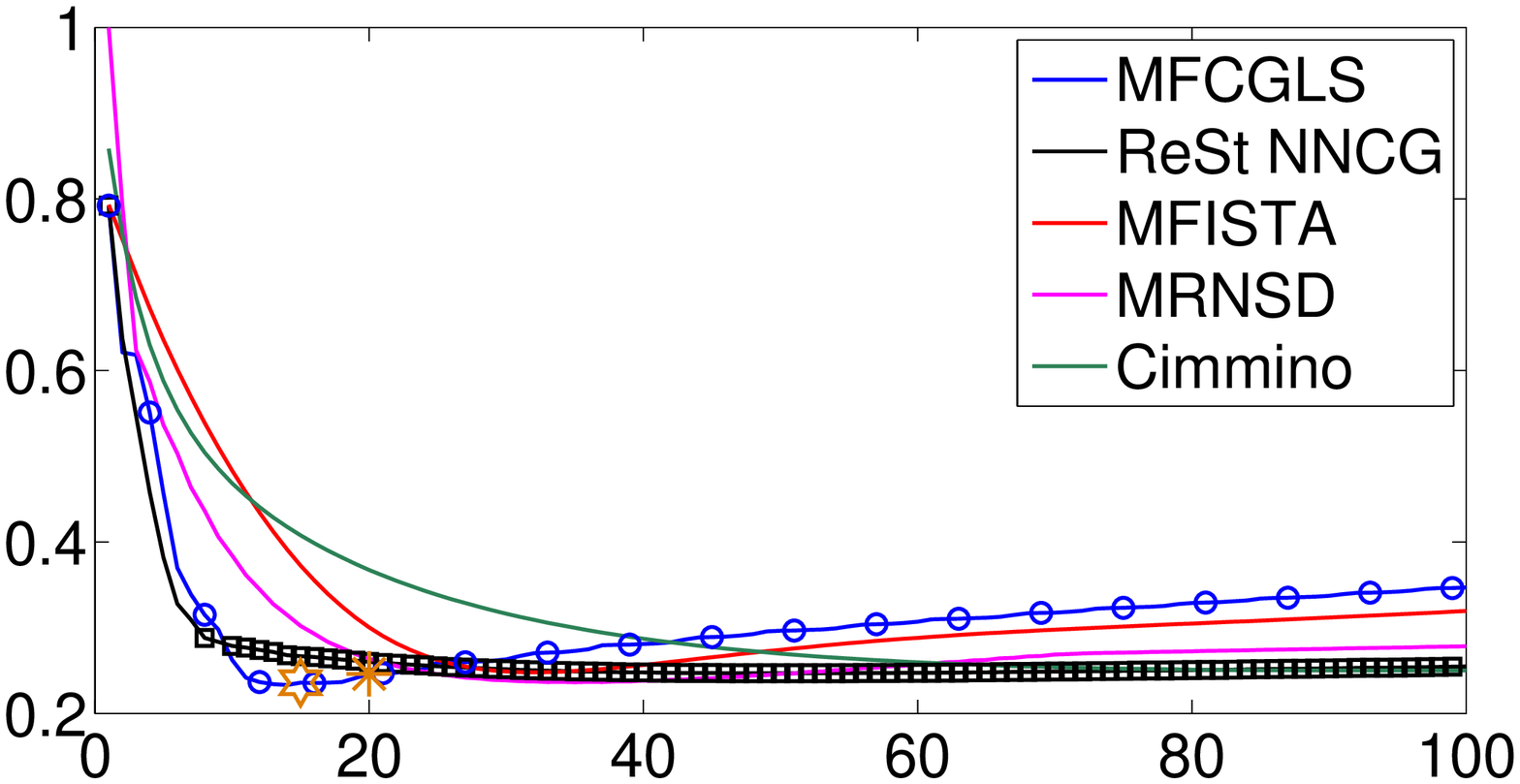}
\end{tabular}
\caption{\emph{\texttt{paralleltomo}} test problem, with $\weps = 5\cdot 10^{-2}$. Upper frame: comparisons of the relative errors obtained when solving the overdetermined problem, with coefficient matrix of size $81088\times 65536$. Lower frame: comparisons of the relative errors obtained when solving the underdetermined problem, with coefficient matrix of size $32580\times 65536$. For the NN-FCGLS and ReSt NNCG methods, a small marker is used to emphasize the iterations where restarts happen. Two big markers highlight the iterations satisfying the first stopping criterion in (\ref{stopC}) (asterisk) and the second stopping criterion in (\ref{stopC}) (hexagram) for NN-FCGLS.}
\label{fig:RelErr_Tomo}
\end{figure}
\begin{figure}[htbp]
\centering
\begin{tabular}{ccc}
\hspace{-0.9cm}{\small \textbf{(a)}} & \hspace{-0.9cm}{\small \textbf{(b)}} & \hspace{-0.9cm}{\small \textbf{(c)}} \\ 
\hspace{-0.9cm}\includegraphics[width=5.8cm]{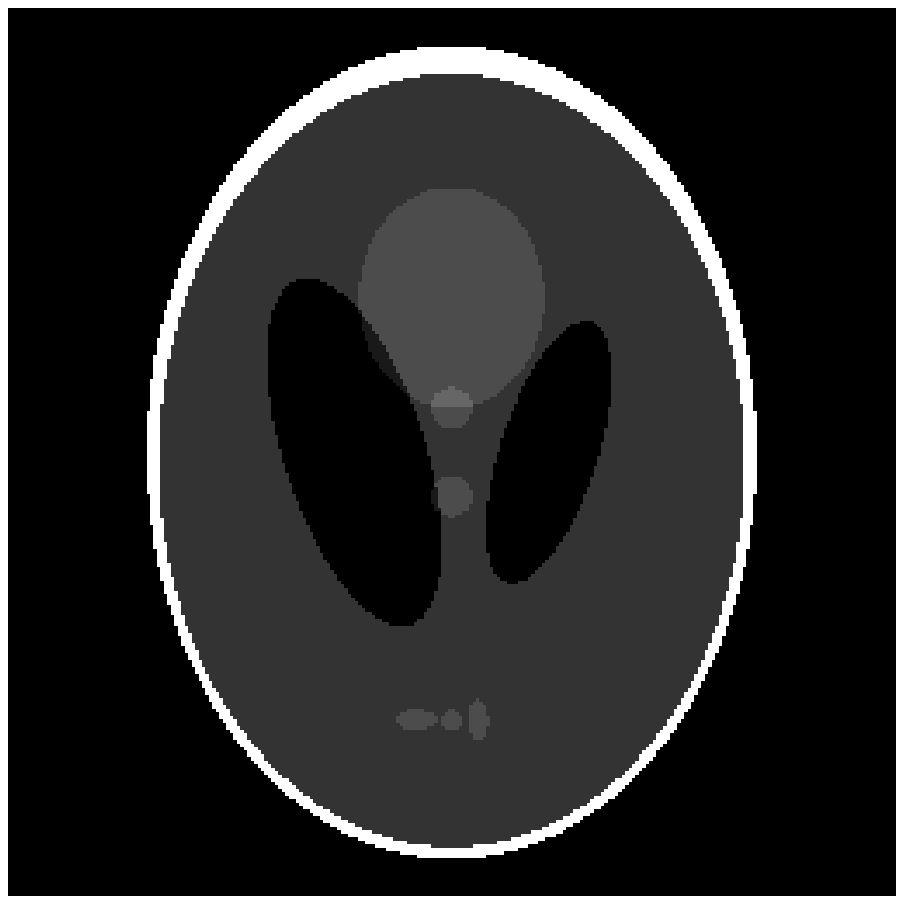} & 
\hspace{-0.9cm}\includegraphics[width=5.8cm]{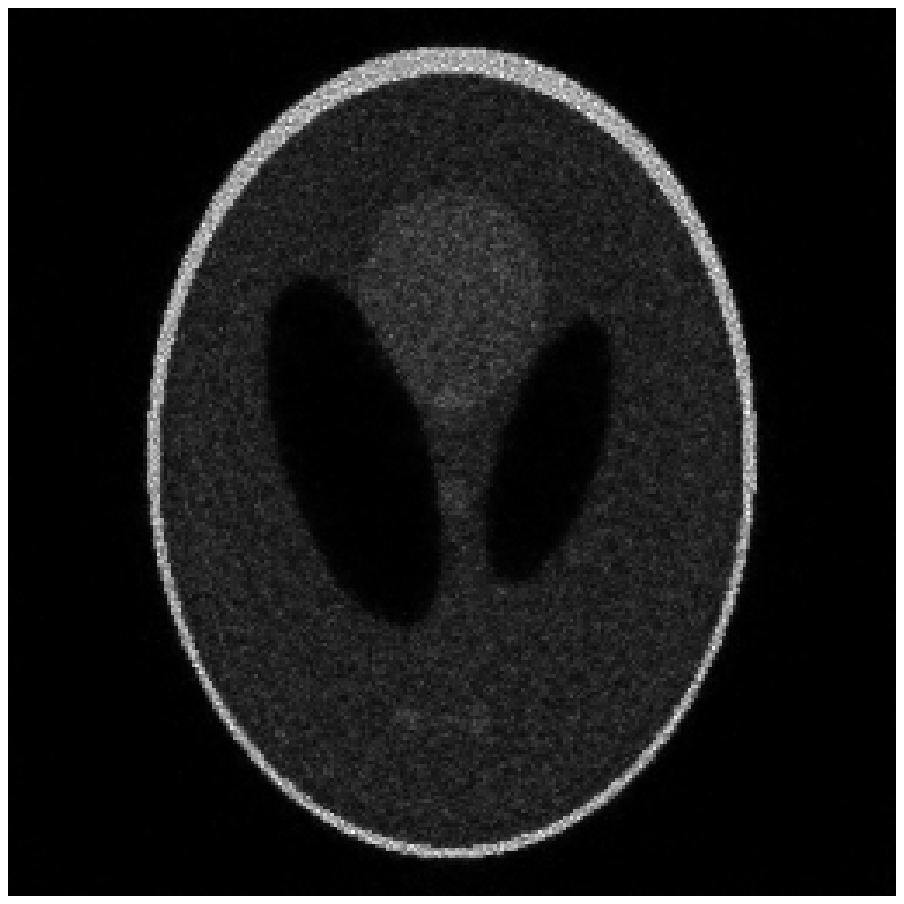} & 
\hspace{-0.9cm}\includegraphics[width=5.8cm]{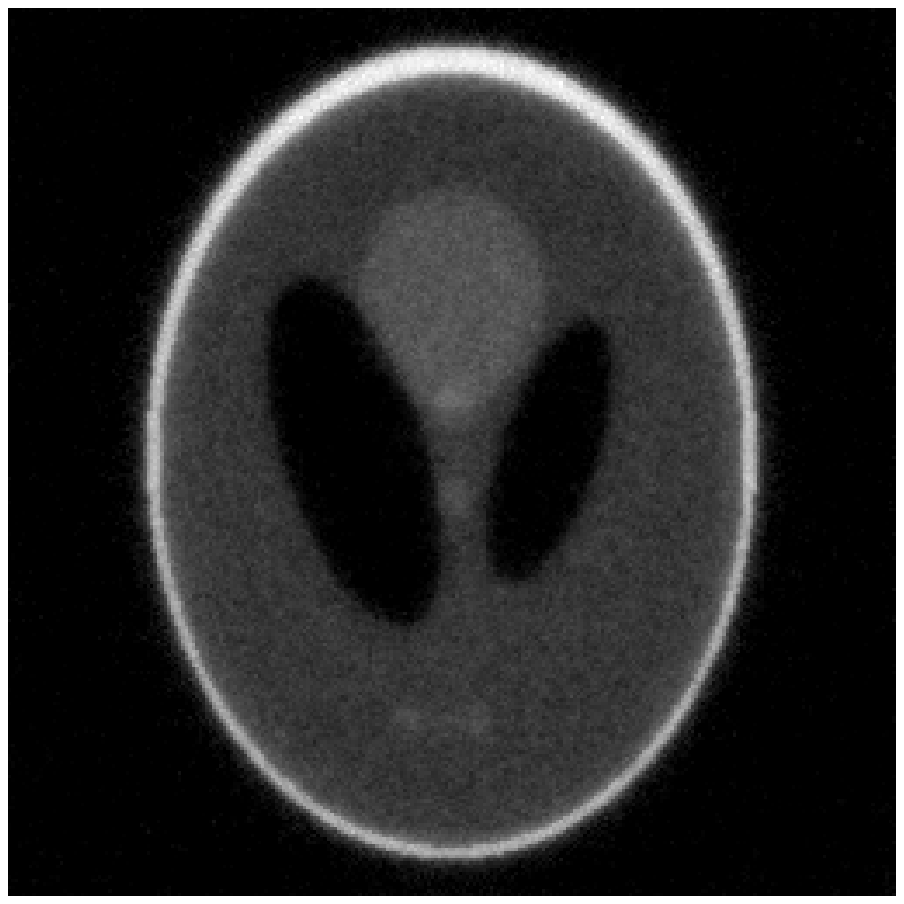}
\end{tabular}%
\caption{Images related to the underdetermined \emph{\texttt{paralleltomo}} test problem, with $\weps = 5\cdot 10^{-2}$. Relative errors are reported within parentheses. \textbf{(a)} Exact image. Restorations obtained at the ${15}$th iteration of the: \textbf{(b)} NN-FCGLS method ${(0.2358)}$, and \textbf{(c)} MFISTA method ${(0.3726)}$.}
\label{fig:Img_Tomo}
\end{figure}
\begin{table}
\footnotesize
\centering
\caption{\emph{\texttt{paralleltomo}} test problem, with $\widetilde{\eps}=5\cdot10^{-2}$. Average values over $10$ runs of the test problem, with different noise realizations.}
\label{tab:Tomo}
\begin{tabular}{|lcccc|}
\hline
 & {\textbf{rel.error}} & {\textbf{iterations}} & {\textbf{tot.time}} & {\textbf{av.time}}\\\hline
 \multicolumn{5}{|c|}{\textbf{overdetermined} (size $81088\times 65536$)}\\\hline
 {\textbf{NN-FCGLS}} & 1.8268e-01 & 17.33 & 0.92 & 0.09\\ \hline
 {\textbf{ReSt NNCG}} & 2.0133e-01 & 56.00 & 16.31 & 0.11\\ \hline
  {\textbf{MFISTA}} & 2.0029e-01 & 37.00 & 53.13 & 1.44 \\ \hline
  {\textbf{MRNSD}} & 1.8506e-01 & 45.00 & 4.10 & 0.09\\ \hline
  {\textbf{Cimmino}} & 1.9982e-01 & 100.00 & 33.47 & 0.33\\ \hline
 \multicolumn{5}{|c|}{\textbf{underdetermined} (size ${32580\times 65536}$)}\\\hline
{\textbf{NN-FCGLS}} & 2.3145e-01 & 13.00 & 0.15 & 0.07\\ \hline
{\textbf{ReSt NNCG}} & 2.4572e-01 & 51.00 & 0.59 & 0.05\\ \hline
{\textbf{MFISTA}} & 2.4634e-01 & 32.00 & 12.41 & 0.39\\ \hline
{\textbf{MRNSD}} & 2.3485e-01 & 35.00 & 3.43 & 0.09\\ \hline
{\textbf{Cimmino}} & 2.4715e-01 & 94.33 & 8.84 & 0.09\\ \hline
\end{tabular}
\end{table}
Looking at the results for the Gaussian and Poisson noise case, one can see that the new methods are inherently faster than the MRNSD-like methods. However, while KWMRNSD has a better performance than WMRNSD for this test problem, CP-NN-FCGLS behaves very similarly to CP-NN-FCGLS($k$).

\textbf{Example 3.}
The last set of experiments uses the \texttt{paralleltomo} test problem \cite{AIRT}. By varying the acquisition parameters, two sparse sensing matrices $A$ are obtained: one of size ${81088\times 65536}$ (overdetermined case), and one of size ${32580\times 65536}$ (underdetermined case). Gaussian noise of level $\weps=5\cdot 10^{-2}$ is added. The parameter $\mmaxin$ of Algorithm \ref{alg:CPNN-FCGLS} is set to 10, the untruncated version of FCGLS is considered, and the parameter $\tau$ appearing in (\ref{stopC}) is set to $10^{-2}$. \cref{tab:Tomo} reports the average results obtained running 10 times the overdetermined and underdetermined problems, with different noise realizations. In addition to the methods considered in the previous examples, also the nonnegative Cimmino method as implemented in \cite{AIRT} is tested. All the methods are stopped after 100 (total) iterations, and a zero initial guess is used (so that, for NN-FCGLS, the identity matrix of order $N$ is taken as $X^{(0)}$).
\cref{fig:RelErr_Tomo} displays the history of the relative errors for the overdetermined and underdetermined cases. Quite interestingly, one can see the ReSt NNCG method to perform slightly better than the NN-FCGLS method during the first iterations. Indeed, when performing the ReSt NNCG method, nonnegativity is imposed only at each restart, and no preconditioning is considered. For this reason, ReSt CGNN is initially faster than NN-FCGLS, but then it rapidly slows down. The NN-FCGLS, MFISTA and MRNSD methods are clearly  \cref{fig:Img_Tomo} shows the reconstructions obtained at the 15th iteration of the NN-FCGLS and MFISTA method applied to the underdetermined problem. 
For this test problem, NN-FCGLS is extremely efficient. Indeed, in both the overdetermined and the underdetermined cases, NN-FCGLS can deliver the best reconstructions in the least number of iterations.

\section{Final remarks and future work}\label{sect:final}
An original, efficient and promising strategy was presented to solve nonnegative linear least squares problems. The new approach is called NN-FCGLS, and it exploits flexible Krylov subspaces methods. To the best of our knowledge, NN-FCGLS is the first systematic attempt to enforce nonnegative approximations within the framework of Krylov subspace methods. The extensive numerical tests displayed in the paper show that the new method is very competitive with other state-of-the-art approaches. Indeed, NN-FCGLS provides faster and better reconstructions with respect to many methods already available in the literature. Moreover, it can be used to solve problems that are affected by both Gaussian and Poisson noise. 

{Although a preliminary theoretical analysis of NN-FCGLS is already provided in this paper}, future work should concentrate on a {deeper theoretical understanding of the regularizing properties of iterative solvers based on flexible Krylov subspaces.} 
Handling additional constraints (e.g., box constraints, sparsity), {considering Krylov methods other than CGLS, }and developing parallel implementations, are interesting and challenging extensions of the present method, which could significantly impact large-scale imaging applications, in particular radio-interferometric imaging in astronomy and magnetic resonance imaging in medicine. 

\section*{Acknowledgements}
We wish to thank the anonymous Referees for providing many insightful remarks that considerably helped us to improve the paper. Silvia Gazzola is very grateful to James Nagy, for engaging in many stimulating discussions about nonnegatively constrained least squares problems over the last years.

\bibliographystyle{plain}
\bibliography{NN}

\end{document}